\documentclass[final]{siamltex}
\usepackage{amsmath, amssymb}

\usepackage{showkeys}
\newtheorem{thm}{Theorem}[section]
\newtheorem{cor}[thm]{Corollary}
\newtheorem{lem}[thm]{Lemma}

\newtheorem{hypothesis}[thm]{Hypothesis}

\newtheorem{prop}[thm]{Proposition}

\newtheorem{defn}[thm]{Definition}

\newtheorem{rem}[thm]{Remark}

\def\bE{\mathbb{E}}
\def\bU{\mathbb{U}}
\def\bD{\mathbb{D}}
\def\cF{\mathcal{F}}
\def\cV{\mathcal{V}}
\def\bR{\mathbb{R}}
\def\P{\mathbb{P}}
\def\PP{\mathbb{P}}

\def\dd{{\rm d}}

\def\tn{|\!|\!|\!}
\def\bx{\mathbf{x}}

\def\bq{\mathbf{q}}
\def\by{\mathbf{y}}
\def\bJ{\mathbb{J}}
\def\cG{\mathcal{G}}
\def\cL{\mathcal{L}}
\def\cB{\mathcal{B}}
\def\cU{\mathcal{U}}
\def\R{\mathbb{R}}
\def\bN{\mathbb{N}}
\def\cK{\mathcal{K}}
\def\bh{\mathbf{h}}
\def\bL{\mathbb{L}}
\def\cH{\mathcal{H}}
\def\cQ{\mathcal{Q}}
\def\ep{\varepsilon}

\author{F. Confortola%
  \thanks{Politecnico di Milano, Dipartimento di Matematica
    F. Brioschi, Piazza Leonardo da Vinci 32, 20133 Milano,
    Italia. {\tt fulvia.confortola@polimi.it}} %
  \and E. Mastrogiacomo%
  \thanks{Politecnico di Milano, Dipartimento di Matematica
    F. Brioschi, Piazza Leonardo da Vinci 32, 20133 Milano,
    Italia. } %
}

\title{Feedback optimal control for stochastic Volterra equations with
  completely monotone kernels}

\begin{document}

\maketitle

\begin{abstract}
    In this paper we are concerned with a class of stochastic Volterra integro-differential problems with completely monotone kernels, where we assume that the noise enters the system when we introduce a control. We start by reformulating the state equation into a semilinear evolution equation which can be treated by semigroup methods. The application to optimal control provide other interesting result and require a precise descriprion of the properties of the generated semigroup. 

The first main result of the paper is the proof of existence and uniqueness of a mild solution for the corresponding Hamilton-Jacobi-Bellman (HJB) equation. The main technical point consists in the differentiability of the BSDE associated with the reformulated equation with respect to its initial datum $x$. 
\end{abstract}



\pagestyle{myheadings}
\thispagestyle{plain}
\markboth{F. Confortola \and E. Mastrogiacomo}{Feedback optimal control for stochastic Volterra equation}

\section{Introduction}\label{sec:intr}
We are concerned with the following optimal control problem for an infinite dimensional  stochastic integral equation 
of Volterra type on a separable Hilbert space $H$:
\begin{align}\label{eq:Volterra}
    \begin{cases}
       \frac{\dd}{\dd t} \int_{-\infty}^t a(t-s)u(s)\dd s= A u(t)+ f(t,u(t))\\
         \qquad \qquad \qquad \qquad + g \,[\, r(t,u(t),\gamma(t))
            + \dot{W}(t)\, ], \qquad t\in [0,T]\\
       u(t)=u_0(t), \qquad t\leq 0.
    \end{cases}
\end{align}
In the above equation $W(t), \ t\geq 0$ is a cylindrical Wiener process defined on a suitable probability space $(\Omega, \cF, (\cF_t)_{t\geq 0},\P) $ (whose properties will be specified later) with values into a (possibly different) Hilbert space $\Xi$; the unknown $u(\cdot)$, representing the state 
of the system, is an $H$-valued process. Also, the control is modellized by the predictable
process $\gamma$ with values in some specified subset $\cU$ (the set of control actions) of a third Hilbert space $U$.

The kernel $a$ is completely monotonic, locally integrable and
singular at $0$; $A$ is a linear operator which generates an analytical semigroup; $g$ is 
bounded linear mapping from $\Xi$ into $H$ and $r$ is a bounded Borel measurable mapping from $[0,T]\times H \times U$ into $\Xi$.
We notice that the control enters the system togheter with the noise.

The optimal control that we wish to treat in this paper consists in minimizing 
a cost functional of the form
\begin{align}\label{eq:cost}
    \bJ(u_0,\gamma)=\bE \int_0^T l(t,u(t),\gamma(t))\dd t + \bE (\phi(u(T))),
\end{align}
where $l$ and $\phi$ are given real-valued functions.

We adopt the semigroup approach based on the complete monotonicity
of the kernel as been initiated in \cite{desch/miller/1988,miller/1974} and recently developed for the stochastic case in \cite{bonaccorsi/desch,BoMa-JEE,BMC11} . Within this approach,
equation \eqref{eq:Volterra} is reformulated into an abstract stochastic evolution 
equation without memory on a different Hilbert space $X$. Namely, we rewrite 
equation \eqref{eq:Volterra} as
\begin{align}\label{eq:state-eq}
   \begin{cases}
    \dd \bx(t)= B \bx(t)\dd t+ (I-B)P \, f(t, J \bx(t))\dd t \\
       \qquad \qquad \qquad + (I-B)P g (r(t,J \bx(t),\gamma(t))\dd t+\dd W(t))\\
     \bx(0)=x.
   \end{cases}
\end{align}
 Here $B$ is the infinitesimal generator of an analytic semigroup $e^{tB}$
on $X$. $P: H \to X$ is a linear mapping which acts as a sort of projection into the space 
$X$. $J: \ D(J) \subset X \to H$ is an unbounded linear functional on $X$,
which gives a way going from $\bx$ to the solution to problem \eqref{eq:Volterra}.
In fact, it turns out that $u$ has the representation 
\begin{align*} 
     u(t)=\begin{cases}
           J \bx(t), & t>0,\\
          u_0(t), & t\leq 0.
            \end{cases}
\end{align*}
For more details, we refer to the original papers \cite{bonaccorsi/desch,homan}.

Further, the optimal control problem, reformulated into the state setting $X$, 
consists in minimizing the cost functional
\begin{align*}
     \bJ(x,\gamma)=\bE \int_0^T l(t,J \bx(t),\gamma(t))\dd t + \bE \phi(J \bx(T))
\end{align*}
(where the initial condition $u_0$ is substituted by $x$ and the process $u$ is substituted by $J\bx$). It follows that $\gamma$ is an optimal control
for the original Volterra equation if and only if it is an optimal control for that state equation
\eqref{eq:state-eq}.  
 
We notice that equation \eqref{eq:state-eq} has unbounded coefficients. Similar stochastic problems are
present in literature (see \cite{DeFuTe,BMC11,dpz:Ergodicity,Masiero,Pa}), also in connection with optimal control. Usually, they arise in a wide variety of applications in physical problems, see the monograph \cite{GriLoSta,pruss} for
some examples, in interacting biological populations and harvesting problems
and in problems in mathematical finance. 
For instance, an example of physically realistic situation which we have in mind is the control of a fluid in the context of thermo-dynamic or fractional diffusion-wave equations. 

Our purpose is not only to prove existence of optimal controls, but mainly to
characterize them by an optimal feedback law. In other words, we wish to perform 
the standard program of synthesis of the optimal control that consists in solving 
the associated forward-backward stochastic differential equation. 
 In our case, this is given by
\begin{align}\label{eq:bsde}
     \begin{cases}
           \dd Y(s) = \psi(s,\bx (s,t,x),Z(s)) \dd s + Z(s)\dd W(s),\qquad s\in [t,T],\\
           Y(T)=\phi(\bx(T,s,x))
     \end{cases}
\end{align}
 where $\psi$ is the hamiltonian function of the control problem, defined in terms of $l$ and $r$, while $\bx(s,t,x)$ stands for the solution of equation \eqref{eq:state-eq} starting at time 
$t$ from $x \in X$.  It is classical that, under suitable assumption
on $l,\, r$ and $\phi$, problem \eqref{eq:bsde} admits a unique (weak) solution.
We set $v(t,x)=Y(t)$. Then, for any $t\in [0,T],x\in X_\eta$, $v(t,x)$ is the value function of the control problem, that is to say, it realizes the average minimal cost ``payed'' by the system starting at time $t$ in $x$. More precisely, $v$ satisfies the so-called  \emph{fundamental relation}:  
\begin{multline}\label{eq:fund-rel-1}
        \bJ(x,\gamma)=v(0,x)+ \bE \int_0^T [\, -\psi(s,\bx(s),\nabla v(s,\bx(s))(I-B)P g )\\
     +\nabla v(s,\bx(s))(I-B)P g \ r(s,J\bx(s),\gamma(s))+ l(s,J\bx(s),\gamma(s)) \,]\ \dd s.
\end{multline}
We notice that the above expression requires $v$ to be G\^ateaux differentiable and 
$\nabla v(t,\bx(t)) (I-B)P g$ to be well-defined. As we will see, to prove these facts, the crucial point is to show that we can give a sense to $\nabla v(t,\bx(t)) (I-B)P g$ and to identify it with the process $Z$ coming from the BSDE associated with the control problem. 
 In fact
we recall that $P $ acts from $H$ into $X_\theta$ and it turns out that $Pg$ does not belong 
to $D(B)$ but only to an interpolation space $(X,D(B))_\theta$, for suitable $\theta \in (0,1)$.    
Hence we are forced to prove that the map $(t,x,h) \mapsto \nabla v(t,\bx)(I-B)^{1-\theta} h$
extends to a continuous map on $[0,T] \times X\times X$.
To do that, we start by proving that this extra regularity holds, in a suitable sense, 
for the forward equation \eqref{eq:state-eq} and then it is conserved if we differentiate (in G\^ateaux sense) the backward equation with respect to the process $\bx$.

On the other hand, showing first that $\bx(\cdot, t,x)$ is regular in Malliavin sense,
we can prove that if the map $(t,x,h) \mapsto \nabla v(t,x) (I-B)^{1-\theta} h$ extends to a continuous 
function on $[0,T] \times X \times X$ then the processes 
$t \mapsto v(t,\bx(t;,x))$ and $W$ admit joint quadratic variation in any interval $[s,\tau]$
and this is given by $\int_t^\tau \nabla v(r,\bx(r;s,x))(I-B)P g\, \dd r$. Then we proceed exploiting the characterization of $\int_t^\tau Z(r) \dd r$ as joint quadratic variation between 
$Y$ and $W$ in $[s,\tau]$. As a consequence, the identification between $Z$ and $\nabla v(t,x) (I-B)P\, g$ 
follows by the definition of $v$.

Once the fundamental relation has been shown, we will be able to construct the optimal feedback law. In fact, equality \eqref{eq:fund-rel-1} immediately implies that for every admissible control $\gamma$ and any initial datum $x$,
we have $\bJ(x,\gamma) \geq v(0,x)$ and $\gamma$ is optimal if and only if the following feedback law holds:
\begin{multline*}
      \psi(t,\bx^\gamma(t),\nabla v(t,\bx^\gamma(t)) \,(I-B) P \, g) \\
       = \nabla v(t,\bx^\gamma(t))\,
      (I-B)P \,g \, r(t,J\bx^\gamma(t),\gamma(t))+ l(t,
        J \bx^\gamma(t),\gamma(t))
\end{multline*}
where $\bx^\gamma$ is the trajectory starting at $x$ and corresponding to the control $\gamma$ (see Corollary \ref{cor:fund-rel}).

%
%

The present paper is a first step of our program. Indeed, we consider a stochastic optimal control problem on finite horizon and under nondegeneracy 
assumptions on the diffusion coefficient $g$. Further, we suppose that $r$ and $l$ are Borel measurable $\Xi$-valued functions
sufficiently smooth in order that the Hamiltonian $\psi$ is Lipschitz continuous 
with respect to $\gamma$.
In this way the  
the corresponding BSDE has sublinear growth in the variable $Z$ and can be exploited, e.g., using the techniques developed in Fuhrman and Tessitore \cite{FuTe/2002} or Confortola and Briand \cite{BrCo}. 
 
We notice that an optimal control problem for stochastic Volterra equations is treated in 
\cite{BMC11}, where the drift term of the equation has a linear growth on the control variable, the cost functional has a quadratic growth, and the control process belongs to the class of square integrable, adapted processes with no bound assumed on it.
The substantial difference, in comparison with the cited paper, consists in the fact that,
at our knowledge, our paper is the first attemp to study existence and uniqueness of solutions for the (HJB) equation corresponding to the Volterra equation \eqref{eq:Volterra} and characterize the optimal control by a feedback law.


The paper is organized as follows: the next section is devoted to notations; in Section \ref{sec:anal-set} we transpose the problem in the infinite dimensional framework; in Section 
\ref{sec:state-eq} we establish the existence result for the uncontrolled equation, while in Section \ref{sec:cont-eq} we study the controlled system. In section \ref{sec:Mall} we study the regularity of the uncontrolled solution, in particular in the sense of Malliavin, while in Section \ref{sec:back}
we will study the BSDE associated to the problem. Finally, in Section \ref{sec:HJB} we will study the corresponding (HJB) equation in order to construct an optimal feedback and an optimal control (see, to this end, Section \ref{sec:cont}). 

\section{Notations and main assumptions}
The norm of an element $x$ of a Banach space $E$ will be denoted by $|x|_E$ or simply
$|x|$ if no confusion is possible. If $F$ is another Banach space, $L(E,F)$ denotes the space 
of bounded linear operators from $E$ to $F$, endowed with the usual operator norm. 

The letters $\Xi,\, H, \, U $ will always denote Hilber spaces. Scalar product is denoted 
$\langle \,\cdot\, , \, \cdot \, \rangle$, with a subscript to specify the space, if 
necessary. All Hilbert space are assumed to be real and separable. 

By a cylindrical Wiener process with values in a Hilbert space $\Xi$, defined on a probability space $(\Omega,\cF,\PP)$, we mean a family $(W(t))_{t\geq 0}$ of linear mappings from
$\Xi$ to $L^2(\Omega)$, denoted $\xi \mapsto \langle \xi, W(t)\rangle$ such that
\begin{enumerate}
     \item for every $\xi \in \Xi, ( \langle \xi, W(t) \rangle )_{t\geq 0}$ is a real 
     (continuous) Wiener process;
     \item for every $\xi_1,\xi_2 \in \Xi$ and $t\geq 0$, $\bE(\langle \xi_1, W(t) \rangle \langle \xi_2, W(t) \rangle) = \langle \xi_1, \xi_2 \rangle$.
\end{enumerate}

$(\cF_t)_{t\geq 0}$ will denote the natural filtration of $W$, augmented with the family 
of $\P$-null sets. The filtration $(\cF_t)_{t\geq 0}$ satisfies the usual conditions. 
All the concepts of measurability for stochastic processes refer to this filtration. 
By $\cB(\Gamma)$ we mean the Borel $\sigma$-algebra of any topological space
$\Gamma$.

In the sequel we will refer to the following class of stochastic processes with values 
in an Hilbert space $K$:
\begin{enumerate}
      \item  $L^p(\Omega; L^2(0,T;K))$ defines, for $T>0$ and $p\geq 1$,  
        the space of equivalence classes of progressively measurable processes 
        $y: \Omega \times [0,T) \to K$, such that 
        $$
               |y|^p_{L^p(\Omega; L^2(0,T;K))} := \bE \left[ \int_0^T |y(s)|^2_K \dd s\right]^{p/2}
                <\infty.
        $$
        Elements of $L^p(\Omega; L^2(0,T;K))$ are identified up to modification. 
      \item $L^p(\Omega; C([0,T];K))$ defines, for $T>0$ and $p\geq 1$,  
        the space of equivalence classes of progressively measurable processes 
        $y: \Omega \times [0,T) \to K$, with continuous paths in $K$, such that 
        the norm 
        $$
               |y|^p_{L^p(\Omega; C([0,T];K))} := \bE \left[ \sup_{t \in [0,T]} |y(t)|^p_K \right]
        $$
        is finite. Elements of $L^p(\Omega; C([0,T];K))$ are identified up to indistinguishability.
        \end{enumerate}
        
        We also recall notation and basic facts on a class of differentiable maps acting 
        among Banach spaces, particularly suitable for our purposes (we refer the reader to
        Fuhrman and Tessitore \cite{FuTe/2002} or Ladas and Lakshmikantham \cite[Section 1.6]{LaL}
         (1970) for details and properties. 
        Let now $X,Y,V$ denote Banach spaces. We say that a mapping $F: X \to V$ 
        belongs to the class $\cG^1(X,V)$ if it is continuous, G\^ateaux differentiable on X,
        and its G\^ateaux derivative $\nabla F: X \to L(X,V)$ is strongly continuous.
        
        The last requirement is equivalent to the fact that for every $h\in X$ the map
        $\nabla F(\cdot)h: X \to V$ is continuous. Note that $\nabla F: X \to L(X,V)$
        is not continuous in general if $L(X,V)$ is endowed with the norm operator 
        topology; clearly, if it happens then $F$ is Fr\'echet differentiable on $X$.
        It can be proved that if $F\in \cG^1(X,V)$ then $(x,h) \mapsto \nabla F(x)h$
        is continuous from $X \times X$ to $V$; if, in addition, $G$ is in $\cG^1(V,Z)$ then
        $G(F)$ is in $\cG^1(X,Z)$ and the chain rule holds: 
        $\nabla(G(F))(x)=\nabla G(F(x))\nabla F(x)$. When $F$ depends on additional 
        arguments, the previous definitions and properties have obvious generalizations.
        In addition to the ordinary chain rule stated above, a chain rule for the Malliavin derivative 
        operator holds: for the reader convenience we refer to Section \ref{sec:Mallreg} for a 
        brief introduction to this subject.

        Moreover, we assume the following. 
       \begin{hypothesis}\label{hp:a,A,f,g,r,W}
              \begin{enumerate}
                 \item The kernel $a: (0,\infty) \to \R$ is completely monotonic, locally
                  integrable, with $a(0+)=+\infty$. The singularity in $0$ shall satisfy some technical
                 conditions that we make precise in Section \ref{sec:anal-set}. 
                 \item $A: D(A) \subset H \to H$ is a sectorial operator in $H$. Thus $A$ generates an
                 analytic semigroup $e^{tA}$. 
                 \item The function $f: \, [0,T] \times H \to H$ is measurable, for every 
                  $t\in [0,T] $ the function $f(t,\, \cdot \,): \ H \to H$ is continuously 
                 G\^ateaux differentiable  
                 and there exist constants $L>0$ and $C>0$ such that 
                 \begin{align*}
                        & |f(t,u)-f(t,v)| \leq L |u-v|, \qquad t\in [0,T], \ u,v \in H;\\
                        & |f(t,0)| + \|\nabla_u f(t,u)\|_{\mathcal{L}(H)} \leq  C, \qquad t \in [0,T], \ u\in H
                 \end{align*}
                 \item $g$ belongs to $L_2(\Xi,H)$, that is to the space of Hilbert-Schmidt operators from $\Xi$ to $H$, endowed with the Hilbert-Schmidt norm $||g||_{L_2(\Xi,H)}^2= {\rm Tr}(g g^*)$.
                 \item The function $r:\ [0,T]\times H \times U \to \Xi$ is Borel measurable for a.e. $t\in [0,T]$ and there 
                  exists a positive constant $C>0$ such that 
                 \begin{align*}
                         & |r(t,u_1,\gamma)-r(t,u_2,\gamma)| \leq C(1+|u_1|+|u_2|)|u_1-u_2|), \qquad
                         u_1,u_2 \in H, \gamma \in U;\\
                         & |r(t,u,\gamma)|_\Xi \leq C, \qquad u\in H, \gamma \in U.
                 \end{align*}
                 \item The process $(W(t))_{t\geq 0}$ is a cylindrical Wiener process defined on 
                 a complete probability space $(\Omega,\cF,(\cF_t)_{t\geq 0},\P)$ with values in the
                 Hilbert space $\Xi$.
              \end{enumerate}
        \end{hypothesis}
        The initial condition satisfies a global exponential bound as well as a linear 
         growth bound as $t\to 0$:
         \begin{hypothesis}\label{hp:dato-in}
              \begin{enumerate}
                  \item There exist $M_1 >0$ and $\omega >0$ such that $|u_0(t)|\leq M_1 e^{\omega t}$ for all $t\leq 0$;
                  \item There exist $M_2>0$ and $\tau >0$ such that $|u_0(t)-u_0(0)|\leq M_2 |t|$ for all 
                 $t \in [-\tau,0]$;
               \item $u_0(0)\in H_{\varepsilon}$ for some $\varepsilon \in (0,1/2)$. 
              \end{enumerate}
          \end{hypothesis}
        Concerning the functions $l$ and $\phi$ appearing in the cost functional we make the
        following general assumptions:
        \begin{hypothesis}\label{hp:l,phi}
           \begin{enumerate}
               \item The functions $l: [0,T] \times H \times U \to \R$ and $\phi: H \to \R$
               are Borel measurable;
                \item There exist a positive constant $C$ and $k\in \bN$
               such that for any $\gamma\in U$ the following
                bound is satisfied
                \begin{align*}
                      & |l(t,u_1,\gamma)-l(t,u_2,\gamma)|\leq C(1+|u_1|+|u_2|)|u_1-u_2|), \qquad u_1,u_2 \in H, \gamma \in U\\
                     &  0 \leq |l(t,0,\gamma)| \leq C.
                \end{align*}
                \item There exists $L>0$ such that, for every $u_1,u_2\in H$ we have
                \begin{align*}
                       |  \phi(u_1)-\phi(u_2)|_H \leq L |u_1-u_2|_H.
               \end{align*}
                 Moreover, $ \phi \in \cG^1(H,\R)$.
           \end{enumerate}
       \end{hypothesis}
        
        We consider the following notion of solution for the Volterra equation \eqref{eq:Volterra}.
        \begin{defn}
              We say that a process $u=(u(t))_{t\geq 0}$ is a solution to equation \eqref{eq:Volterra}
             if $u$ is an adapted, $p$-mean integrable, continuous 
             $H$-valued predictable process and the identity
              \begin{multline*}
                    \int_{-\infty}^t a(t-s)\langle u(s),\zeta\rangle_H \dd s = 
                     \langle \bar{u},\zeta \rangle_H + \int_0^t \langle u(s),A^\star\zeta\rangle_H  \dd s \\
                      +\int_0^t \langle f(s,u(s)), \zeta \rangle \dd s + \int_0^t \langle g \, r(s,u(s),\gamma(s)),\zeta \rangle_H \dd s
                      + \langle g W(t),\zeta\rangle  
              \end{multline*}
           holds $\P$-a.s. for arbitrary $t\in [0,T]$ and $\zeta \in D(A^\star)$,  with $A^\star$ being the adjoint of the operator $A$
           and 
           \begin{align*}
                   \bar{u}=\int_{-\infty}^0 a(-s) u_0(s) \dd s.
           \end{align*}
        \end{defn}

\section{The analitical setting}\label{sec:anal-set}
     A completely monotone kernel $a: (0,\infty) \to \R$ is a continuous, monotone decreasing function, infinitely often derivable, such that
$$
       (-1)^n \frac{\dd^n }{\dd t^n}a(t) \geq 0, \quad t\in (0,\infty), \ n=0,1,2,\dots
$$   
By Bernstein's theorem, $a$ is completely monotone if and only if there exists a positive measure 
   $\nu $ on $[0,\infty)$ such that
  $$
      a(t)=\int_{[0,\infty)} e^{-\kappa t}\nu(\dd \kappa), \quad t>0.
$$
Under the assumpion $a\in L^1(0,1)$, it holds that the Laplace trasform $\hat{a}$ is well defined and it is given in terms of $\nu$ by
\begin{align*}
      \hat{a}(s)=\int_{[0,\infty)} \frac{1}{s+\kappa}\nu(\dd \kappa).
\end{align*}
We introduce the quantity
$$
      \alpha(a)= \sup\left\{\rho \in \R: \ \int_c^\infty s^{\rho - 2}\frac{1}{\hat{a}(s)}\dd s <\infty\right\}
$$
and we make the following assumption:
\begin{hypothesis}
     $\alpha(a)>1/2$.
\end{hypothesis}

\begin{rem}
It is known from the theory of deterministic Volterra equations that the singularity 
of $a$ helps smoothing the solution. We notice that $\alpha(a)$ is independent on the choice of $c>0$ and this quantity 
describes the behaviour of the kernel near $0$; by this way we ensure that smoothing is suffiecient to keep
the stochastic term tractable. 
\end{rem}

It is known that we can associate to any completely monotone kernel
$a$, by means of Bernstein's Theorem \cite[pag. 90]{pruss}, a measure
$\nu$ on $[0,+ \infty)$ such that
\begin{equation}
  \label{eq:Bernstein}
  a(t) = \int_{[0,+ \infty)} e^{-\kappa t}\, \nu({\rm d}\kappa).
\end{equation}
From the required singularity of $a$ at $0+$ we obtain that \(\nu([0,+
\infty))=a(0+)=+ \infty\) while for $s>0$ the Laplace transform $\hat
a$ of $a$ verifies
\begin{align*}
  \hat a(s) = \int_{[0,+ \infty)} \frac 1{s+\kappa} \, \nu({\rm
    d}\kappa) < + \infty.
\end{align*}

Under the assumption of complete monotonicity of the kernel, a
semigroup approach to a type of abstract integro-differential
equations encountered in linear viscoelasticity was introduced in
\cite{desch/miller/1988} and extended  to the case of Hilbert space valued equations in \cite{bonaccorsi/desch}. In order to simplify the exposition we quote from \cite{bonaccorsi/desch} the main result concerning the derivation of the state equation  \eqref{eq:state-eq}.

We will see that  this approach allow to treat the case of semilinear, stochastic integral equations; we start for simplicity with the equation
\begin{equation}
  \label{eq:ev-f}
  \begin{aligned}
    \frac{d}{dt} \int_{-\infty}^t a(t-s)u(s) \, {\rm d}s &= Au(t) +
    f(t), \qquad& t \in [0,T]
    \\
    u(t) &= u_0(t), \qquad& t \le 0,
  \end{aligned}
\end{equation}
where $f$ belongs to $L^1(0,T;X)$.
The starting point is the following identity, which follows by
Bernstein's theorem
\begin{align*}
  \int_{-\infty}^t a(t-s)u(s)\, {\rm d}s = \int_{-\infty}^t
  \int_{[0,+ \infty)} e^{-\kappa(t-s)} \, \nu({\rm d}\kappa)\, u(s) {\rm
    d}s = \int_{[0,+ \infty)} \bx(t,\kappa)\, \nu({\rm d}\kappa)
\end{align*}
where we introduce the state variable
\begin{equation}
  \label{eq:intro_v_defined}
  \bx(t,\kappa) = \int_{-\infty}^t e^{-\kappa(t-s)}u(s)\, {\rm d}s.
\end{equation}
Formal differentiation yields
\begin{equation}
  \label{eq:intro_v_diff_equation}
  \frac{\partial}{\partial t} \bx(t,\kappa) = -\kappa \bx(t,\kappa) +
  u(t),
\end{equation}
while the integral equation (\ref{eq:ev-f}) can be rewritten
\begin{equation}
  \label{eq:intro_rewritten}
  \int_{[0,+ \infty)} (-\kappa \bx(t,\kappa)+u(t))\, \nu({\rm d}\kappa) =
  A u(t) + f(t).
\end{equation}
Now, the idea is to use equation \eqref{eq:intro_v_diff_equation} as
the state equation, with $B \bx= -\kappa \bx(\kappa) + u$, while
\eqref{eq:intro_rewritten} enters in the definition of the domain of
$B$.

In our setting, the function \(\bx(t,\cdot)\) will be considered the
state of the system, contained in the state space $X$
that consists of all Borel measurable functions $
\by: [0,+ \infty) \to H$ such that the seminorm
\begin{align*}
  \Vert \tilde\bx \Vert_X^2 := \int_{[0,+ \infty)} (\kappa + 1) \vert 
\by(\kappa)\vert_H^2 \, \nu({\rm d}\kappa)
\end{align*}
is finite. We shall identify the classes $\by$ with respect to
equality almost everywhere in $\nu$.

Let us consider the initial condition. We introduce the space 
$$
   \tilde X_0:=\left\{ u: \R_{-} \to H: \ exists \ M>0 \ and \ \omega>0 \ s.t. \ |u(t)|<Me^{\omega t}, \ t\leq 0 \right\}$$
and we endow it with a  positive inner product
\begin{align*}
  \langle u,v \rangle_{\tilde X} = \int \int [a(t+s) - a'(t+s)]
  \langle u(-s), v(-t) \rangle_H \, {\rm d}s \, {\rm d}t;
\end{align*}
then, setting $\tilde N_0 = \{ u \in \tilde X_0\ :\ \langle u,u
\rangle_{\tilde X} = 0\}$, $\langle \cdot,\cdot \rangle_{\tilde X}$ is
a scalar product on $\tilde X_0 / \tilde N_0$; we define $\tilde X$
the completition of this space with respect to $\langle \cdot,\cdot
\rangle_{\tilde X}$. We let the operator $Q: \tilde X \to X$ be given
by
\begin{equation}
  \label{eq:intro_initial}
  \bx(0,\kappa)= Q u_0(\kappa) = \int_{-\infty}^0 e^{\kappa s}u_0(s)\, {\rm d}s.
\end{equation}

It has been proved in \cite[Proposition 2.5]{bonaccorsi/desch} that the
operator $Q$ is an isometric isomorphism between $\tilde X$ and $X$.
This operator maps the initial value of the stochastic Volterra
equation in the initial value of the abstract state
equation. Different initial conditions of the Volterra equation
generate different initial conditions of the state equation.

Hypothesis \ref{hp:dato-in} is necessary in order
to have a greater regularity on the inial value of the state
equation. In fact in this case \cite[Proposition 2.20]{bonaccorsi/desch}) shows that $Q u_0$ belongs to $X_{\eta}$ for
$\eta \in (0,\frac12)$.

\begin{rem}
  We stress that under our assumptions we are able to treat, for
  instance, initial conditions for the Volterra equation of the
  following form
  \begin{align*}
    u_0(t) =
    \begin{cases}
      0, & ( - \infty, -\delta); \\
      \bar{u} & [-\delta,0]
    \end{cases}
  \end{align*}
    provided $\bar{u}$ has a suitable regularity.
\end{rem}

\smallskip

We quote from \cite{bonaccorsi/desch} the main result concerning the
state space setting for stochastic Volterra equations in infinite
dimensions.

\begin{thm}[State space setting]
  \label{t:state space setting}
  Let \(A\), \(a\), \(\alpha(a)\), \(W\) be given above; choose
  numbers \(\eta \in (0,1)\), \(\theta \in (0,1)\) such that
  \begin{align}\label{eq:eta-theta}
    \eta > \frac 12\, (1-\alpha(a)), \quad \theta < \frac 12\,
    (1+\alpha(a)), \quad \theta-\eta>\frac 12.
  \end{align}
  Then there exist
  \hfill\begin{itemize}
  \item[1)] a separable Hilbert space \(X\) and an isometric
    isomorphism \(Q: \tilde X \to X\),
  \item[2)] a densely defined sectorial operator \(B:D(B) \subset X
    \to X\) generating an analytic semigroup \(e^{tB}\) with growth
    bound $\omega_0$,
  \item[3)] its real interpolation spaces \(X_{\rho} =
    (X,D(B))_{(\rho,2)}\) with their norms \(\Vert \cdot
    \Vert_{\rho}\),
  \item[4)] linear operators \(P: H \to X_{\theta}\), \(J:X_{\eta} \to
    H\)
  \end{itemize}
  such that the following holds:

   For each \(x_0 \in X\), the problem (\ref{eq:ev-f}) is
  equivalent to the evolution equation
    \begin{equation}
      \label{pb:evo-state-f}
      \begin{aligned}
        \bx'(t) &= B \bx(t) + (I-B)P f(t)
        \\
        \bx(0) &= x_0
      \end{aligned}
    \end{equation}
    in the sense that if \(u_0\in \tilde X_0\) and \(\bx(t;x_0)\) is the
    weak solution to Problem~\eqref{pb:evo-state-f} with \(x_0 = Qu_0\),
    then \(u(t;u_0) = J\bx(t;x_0)\) is the unique weak solution to
    Problem~\eqref{eq:ev-f}.
\end{thm}

It is remarkable that $B$ generates an analytic semigroup, since in
this case we have at our disposal a powerful theory of optimal
regularity results. In particular, besides the interpolation spaces
$X_\theta$ introduced in Theorem \ref{t:state space setting}, we may
construct the extrapolation space $X_{-1}$, i.e., a Sobolev space of
negative order associated to $e^{t B}$. 

Assume for simplicity that $B$ is of negative type (otherwise, one may
consider $B - \omega_0$ instead of $B$ in the following discussion).
The semigroup $e^{t B}$ extends to $X_{-1}$ and the generator of this
extension, that we denote $B_{-1}$, is the unique continuous extension
of $B$ to an isometry between $X$ and $X_{-1}$. See for instance
\cite[Definition 5.4]{Engel2000} for further details.

\begin{rem}
  \label{re:2.**}
  In the sequel, we shall always denote the operator with the letter
  $B$, even in case where formally $B_{-1}$ should be used
  instead. This should cause no confusion, due to the similarity of
  the operators.
\end{rem}

\section{The state equation: existence and uniqueness}\label{sec:state-eq}
       In this section, motivated by the construction in Section \ref{sec:anal-set}, we shall establish 
      existence and uniqueness result for the following stochastic controlled
Cauchy problem on the space $X$ defined in Section \ref{sec:anal-set}:
      \begin{align}\label{eq:state-eq-2}
   \begin{cases}
    \dd \bx(t)= B \bx(t)\dd t+(I-B)Pf(t,J\bx(t))\dd t+ \\
         \qquad \qquad 
           (I-B)P\,r(t,J \bx(t),\gamma(t))\dd t+ (I-B)P g\,\dd W(t)\\
     \bx(s)=x.
   \end{cases}
   \end{align}
    for $0 \leq s\leq t \leq T$ and initial condition $x\in X_\eta$. The above expression is only formal since the coefficients do not belong to the state space; however, we can give a meaning to the mild form of the equation: 
\begin{defn} 
We say that a continuous, $X$-valued, predictable process $\bx=(\bx(t))_{t\geq 0}$ is a (mild) solution of the state equation \eqref{eq:state-eq-2} if $\P$-a.s.,
   \begin{multline*}
        \bx(t)=e^{(t-s)B}x+\int_s^t e^{(t-\sigma)B}(I-B)Pf(\sigma,J\bx(\sigma))\dd \sigma\\
          \quad +\int_0^t e^{(t-\sigma)B}(I-B)P \, r(\sigma,J\bx(\sigma),\gamma(\sigma))\dd \sigma+ \int_s^t e^{(t-\sigma)B}(I-B)P g\,\dd W(\sigma).
   \end{multline*}
\end{defn}
   Let us state the main existence result for the solution of equation \eqref{eq:state-eq-2}.
   
   \begin{thm}\label{thm:exuni}
        Under Hypothesis \ref{hp:a,A,f,g,r,W}, \ref{hp:dato-in}, for an arbitrary predictable process $\gamma$ with values in $\cU$, for every $0\leq s \leq t\leq T$ and $x\in X_\eta$, there exists a unique adapted process $\bx \in L^p_\cF(\Omega,C([s,T];X_\eta))$ solution of \eqref{eq:state-eq-2}. 
Moreover,
 the estimate
   \begin{align}\label{eq:exuni}
       \bE \sup_{t\in[s,T]} ||\bx(t)||^p_\eta \leq C (1+||x||_\eta^p)
\end{align}
holds for some positive constant $C$ depending on $T$ and the parameters of the problem.
   \end{thm}

\begin{proof}
       The proof of the above theorem prooceds, basically, on the same lines as the proof of
       Theorem 
        $3.2$ in Bonaccorsi and Mastrogiacomo \cite{BoMa-JEE} (2009). 
      First, we define a mapping $\cK$ from $L^p(\Omega;C([0,T];X_\eta))$ to itself by the formula
      \begin{align}\label{eq:kappa}
           \cK(\bx)(t):=e^{(t-s)B}x+ \Lambda(\bx)(t)+\Delta(\bx)(t)+\Gamma(t),
      \end{align}
      where the second, third and last term in the right side of \eqref{eq:kappa} are given by
     \begin{align}
      &\Lambda(\bx)(t)=\int_s^t e^{(t-\tau)B}(I-B)Pf(\tau,J\bx(\tau))\dd \tau \label{eq:lambda} \\
      &\Delta(\bx)(t)=\int_s^t e^{(t-\tau)B}(I-B)P g \, r(\tau,J \bx(\tau),\gamma(\tau)) \dd \tau\\
      &\Gamma(t)= \int_s^t e^{(t-\tau)B}(I-B)P g\,\dd W(\tau) \label{eq:gamma}
      \end{align}
      Then, we will prove that the mapping $\cK$ is a contraction 
      on $L^p(\Omega;C([0,T];X_\eta))$ with respect to the equivalent norm
      $$
             \tn \bx\tn^p_\eta := \bE \sup_{t\in [0,T]} e^{-\beta p t}||\bx(t)||^p_\eta,
       $$
      where $\beta>0$ will be chosen later. For semplicity we fix the initial time 
      $s=0$ and write $\Lambda(t)$ instead of $\Lambda(\bx)(t)$.  

     Our first step is to prove that $\Gamma, \Delta$ and $\Lambda$ are well-defined
     mappings on the space $L^p(\Omega;C([0,T];X_\eta))$ and to give estimates on their norm.
     We choose $\delta$ small enough such that $ 1+\eta-\theta +1/p<\delta<\!
     <1/2$ and define
      \begin{align*}
            y_\eta(\tau) := \int_0^t (t-\sigma)^{-\delta}e^{(t-\sigma)B}(I-B)^{\theta}P\,g \dd W(\sigma). 
      \end{align*}
      Since  the semigroup $e^{tB}$ is analytic, $P$ maps $H$ into $X_{\theta}$ and $g\in L_2(\Xi,H)$, an application of Lemma 7.2 in \cite{LaL}, yields:
     $\bE\int_0^T |y_\eta(\sigma)|^p <\infty$. In particular 
     $y\in L^p([0,T];X)$, $\PP$-a.s. Moroever, if we set
       $$
     ( R_\delta\phi)(t) =\int_0^t (t-\sigma)^{\delta-1} e^{(t-\sigma)B} (I-B)^{1+\eta-\theta}\phi(\sigma) \dd \sigma,
     $$
       then in\cite[Proposition A.1.1.]{dpz:Stochastic} it is proved that $R_\delta$ is a bounded linear operator 
     from $L^p([0,T];X)$ into $C([0,T];X)$. Finally, by stochastic Fubini-Tonelli 
     Theorem we can rewrite:
     \begin{align*}
          (R_\delta y_\eta)(t)&= \int_0^t \int_0^\tau (t-\tau)^{\delta -1}(\tau-\sigma)^{\delta}\\
        &\qquad \qquad \qquad (I-B)^{\eta+1} e^{(t-\sigma)B}P\, g \dd W(\sigma) \ \dd \tau \\
       &=\left(\int_0^1 (1-\tau)^{\delta-1}\tau^{-\delta}\dd \tau\right)
       (I-B)^{\eta} \Gamma(t)
      \end{align*}
     and conclude that $\Gamma(t)\in L^p(\Omega, C([0,T];X_\eta))$. 

     In a similar (and easier) way it is possible to show that $\Lambda(\cdot,t)$ and $\Delta(\cdot,t)$ belong
     to $L^p(\Omega, C([0,T];X_\eta))$. 
     Finally, we conclude that $\cK$ maps $L^p(\Omega;C([0,T];X_\eta))$
     into itself; to this end it is sufficient to recall that the initial condition $x$ belongs 
     to $X_\eta$; but this follows immediately from the analiticity of the
     semigroup, provided that $e^{tB}$ is extended to a constant for $t<s$:
    $$   
         e^{(t-s)B}x=x \quad \textrm{for } t<s.
   $$

    Now we claim that $\cK$ is a contraction in $L^p(\Omega, C([0,T];X_\eta))$. In fact,
    by straightforward estimates we can write
    \begin{align*}
        \tn \Lambda (\bx)(t) -\Lambda(\by)(t)\tn^p_\eta \leq C_{L,T} \beta^{1/2+\delta-(\theta-\eta)} 
     \tn \bx -\by \tn_\eta^p.
    \end{align*}
     Therefore $\cK$ is Lipschitz continuous from $L^p(\Omega,C([0,T];X_\eta))$ into itself;
further, we can find $\beta$ large enough such that $C_{L,T}(2\beta)^{1/2+\delta+\eta-\theta}<1$.
 Hence $\cK$ becomes a contraction on the time interval $[0,T]$ and by a classical 
     fixed point argument we get that there exists  a unique solution of the equation
     \eqref{eq:state-eq-2} on $[0,T]$.
\end{proof}

   \begin{rem}
       In the following it will be also useful to consider the uncontrolled version of 
     equation \eqref{eq:state-eq-2}, namely:
     \begin{align}\label{eq:un-state-eq}
    \begin{cases}
    \dd \bx(t)= B \bx(t)\dd t+(I-B)Pf(t,J\bx(t))\dd t+ (I-B)P g\,\dd W(t)\\
     \bx(0)=x.
   \end{cases}
    \end{align}
    We will refer to \eqref{eq:un-state-eq} as the forward equation.
      We then notice that existence and uniqueness for the above equation can be treated in an identical 
     way as in the proof of Theorem \ref{thm:exuni}. 
   \end{rem}

\section{The controlled stochastic Volterra equation}\label{sec:cont-eq}

As a preliminary step for the sequel, we state two results of existence and uniqueness for (a special case of) the original Volterra
equation. The proofs can be found in \cite[Section 2]{BCM11}.

\begin{prop} 
  The linear equation
  \begin{equation}
    \label{eq:ev-lin}
    \begin{aligned}
      \frac{d}{dt} \int_{-\infty}^t a(t-s)u(s) \, {\rm d}s &= A
      u(t), \qquad t \in [0,T]
      \\
      u(t) &= 0, \qquad t \le 0.
    \end{aligned}
  \end{equation}
  has a unique solution $u \equiv 0$.
\end{prop}

%


\smallskip

Now we deal with existence and uniqueness of the Stochastic Volterra
equation with non-homogeneous terms. To this end we extend the result
in \cite[Theorem 3.7]{bonaccorsi/desch} where the case $f(t) \equiv 0$
is treated.  

\smallskip

\begin{prop}
  \label{te:4.2}
  In our assumptions, let $x_0 \in X_\eta$ for some
  $\frac{1-\alpha(a)}{2} < \eta < \frac{1}{2} \alpha(a)$.  Given the
  process
  \begin{align}
    \label{eq:v-lineare}
    \bx(t) = e^{t B}x_0 + \int_0^t e^{(t-s)B} (I-B) P f(s) \, {\rm d}s +
    \int_0^t e^{(t-s)B} (I-B) P g  \, {\rm d}{W(s)}
  \end{align}
  we define the process
  \begin{equation}
    \label{eq:definition_of_u}
    u(t) = \begin{cases}
      J \bx(t), & t \ge 0, \\
      u_0(t), & t \le 0.
    \end{cases}
  \end{equation}
  Then $u(t)$ is a weak solution to problem
  \begin{equation}
    \label{eq:u-lineare}
    \begin{aligned}
      \frac{d}{dt} \int_{-\infty}^t a(t-s)u(s) \, {\rm d}s &= A u(t) +
      f(t)+ g\dot W(t), \qquad t \in [0,T]
      \\
      u(t) &= u_0(t), \qquad t \le 0.
    \end{aligned}
  \end{equation}
\end{prop}

After the preparatory results stated above, here we
prove that main result of existence and uniqueness of solutions of the
original controlled Volterra equation
\eqref{eq:Volterra}. 

\begin{thm} 
  \label{sol-Vol-contr}
 Assume Hypothesis \ref{hp:a,A,f,g,r,W} and \ref{hp:dato-in}. Let $\gamma$ be an admissible control and $\bx$ be the solution to problem
  (\ref{eq:state-eq})) (associated with $\gamma$) whose existence is proved in Theorem
  \ref{thm:exuni}.
  Then the process
  \begin{equation}
    \label{eq:u}
    u(t) =
    \begin{cases}
      u_0(t), & t \le 0
      \\
      J \bx(t), & t \in [0,T]
    \end{cases}
  \end{equation}
  is the unique solution of the stochastic Volterra equation
  \begin{equation}
    \label{eq:ev-sec4}
    \begin{aligned}
      \frac{d}{dt} \int_{-\infty}^t a(t-s)u(s) \, {\rm d}s &= A u(t) +
      f(t,u(t)) + g\, \left[r(t,u(t),\gamma(t)) + \dot W(t)\right],
      \qquad t \in [0,T]
      \\
      u(t) &= u_0(t), \qquad t \le 0.
    \end{aligned}
  \end{equation}
\end{thm}
  
\begin{proof}
  We propose to fulfill the following steps: first, we prove that the affine equation
    \begin{equation}
      \label{eq:ev-lin-2}
      \begin{aligned}
        \frac{d}{dt} \int_{-\infty}^t a(t-s)u(s) \, {\rm d}s &= A u(t)
        + f(t,\tilde{u} (t)) + g \,\left[r(t,\gamma (t), \tilde u(t)) + \dot
          W(t) \right], \qquad t \in [0,T]
        \\
        u(t) &= u_0(t), \qquad t \le 0.
      \end{aligned}
    \end{equation}
    defines a contraction mapping $\cQ: \tilde u \mapsto u$ on the
    space $L^2_\cF(\Omega;C([0,T];H))$. Therefore, equation
    (\ref{eq:ev-lin-2}) admits a unique solution.
  
   Then we show that the process $u$ defined in (\ref{eq:u})
    satisfies equation (\ref{eq:ev-lin-2}). Accordingly, by the
    uniqueness of the solution, the thesis of the theorem follows.

  \textit{First step.}
We proceed to define the mapping
\begin{align*}
  \cQ: L^p(\Omega;C([0,T];H)) \to L^p(\Omega;C([0,T];H))
\end{align*}
where $\cQ(\tilde u) = u$ is the solution of the problem
\begin{equation}
  \label{eq:ev-cL}
  \begin{aligned}
    \frac{d}{dt} \int_{-\infty}^t a(t-s)u(s) \, {\rm d}s &= A u(t) +
    g(t,\tilde u(t)) \, \left[r(t,\tilde u(t),\gamma(t)) +  \dot W(t)\right],
    \qquad t \in [0,T]
    \\
    u(t) &= u_0(t), \qquad t \le 0.
  \end{aligned}
\end{equation}


  It follows from the uniqueness of the solution, proved in Proposition \ref{eq:ev-lin},
  that the solution $u_i(t)$ ($i=1,2$) has the representation
  \begin{align*}
    u_i(t) =
    \begin{cases}
      J v_i(t), & t \in [0,T]
      \\
      u_0(t), & t \le 0
    \end{cases}
  \end{align*}
  where
  \begin{multline*}
    v_i(t) = e^{t B}v_0 + \int_0^t e^{(t-s)B} (I-B) P g(s, \tilde
    u_i(s)) r(s, \tilde u_i(s), \gamma(s)) \, {\rm d}s
    \\
    + \int_0^t e^{(t-s)B} (I-B) P g(s, \tilde u_i(s)) \, {\rm
      d}{W(s)}.
  \end{multline*}
  In particular,
  \begin{align*}
    U(t) = u_1(t) - u_2(t) =
    \begin{cases}
      J(v_1(t) - v_2(t)), & t \in [0,T]
      \\
      0, & t \le 0;
    \end{cases}
  \end{align*}
  then
  \begin{align*}
    \bE \sup_{t \in [0,T]} e^{-\beta p t} |U(t)|^p \le
    \|J\|^p_{L(X_\eta,H)} \bE \sup_{t \in [0,T]} e^{-\beta p t} \|v_1(t)
    - v_2(t)\|^p_\eta
  \end{align*}
  the quantity on the right hand side can be treated as in Theorem
  \ref{thm:exuni} and the claim follows.

$ $

\textit{Second step}

It follows from the previous step that there exists at most a 
unique solution $u$ of problem (\ref{eq:ev-lin-2}); hence it only remains to prove the representation
formula (\ref{eq:u}) for $u$.

Let $\tilde f(t) = f(t, J \bx(t)) + g r(t, J \bx(t), \gamma(t))$; it is a consequence of Proposition
\ref{te:4.2} that $u$, defined in (\ref{eq:u}), is a weak solution
of the problem
\begin{equation}
  \label{eq:III-1}
  \begin{aligned}
    \frac{d}{dt} \int_{-\infty}^t a(t-s)u(s) \, {\rm d}s &= A u(t) +
    \tilde f(t) +  g  \dot W(t),
    \qquad t \in [0,T]
    \\
    u(t) &= u_0(t), \qquad t \le 0,
  \end{aligned}
\end{equation}
and the definition of $\tilde f$ implies that $u$ is a
weak solution of
\begin{equation}
  \label{eq:III-2}
  \begin{aligned}
    \frac{d}{dt} \int_{-\infty}^t a(t-s)u(s) \, {\rm d}s &= A u(t) +
    f(t,J \bx(t)) + g \left[ r(t, J \bx(t), \gamma(t)) + \dot W(t)
    \right], \ \ t \in [0,T]
    \\
    u(t) &= u_0(t), \qquad t \le 0,
  \end{aligned}
\end{equation}
that is problem (\ref{eq:ev-lin-2}).
\end{proof}


\section{The forward SDE}\label{sec:Mall}
In the following we are concerned with smoothness properties of the forward equation, i.e. of the uncontrolled state equation 
\eqref{eq:un-state-eq} on the time interval $[s,T]$ with initial condition 
$x\in X_\eta$. It will be denoted by $\bx(t;s,x)$, to stress dependence on the initial
data $t$ and $x$. Also, we extend $\bx(\cdot;s,x)$ letting
$\bx(t;s,x)=x$ for $t\in [0,s]$.

Before proceeding with the program mentioned above, we list relevant properties of the
nonlinear term of the reformulated equation. 

\begin{rem}
   It follows directly by the properties of the nonlinear mapping $f$ and the operator $J$ that, under Hypothesis \ref{hp:a,A,f,g,r,W} the function $(t,x) \mapsto f(t,Jx)$ from
   $[0,T]\times X_\eta$ into $H$ is measurable and it verifies the following estimates 
     \begin{align*}
        & |f(t,Jx)-f(t,Jy)|_H\leq L \|J\|_{L(X_\eta;H)}\|x-y\|_\eta \qquad t\in [0,T], \ x,y\in X_\eta;\\
        &  |f(t,J 0)|\leq C, \qquad t\in [0,T]. 
     \end{align*}
     Moreover, for every $t\in [0,T]$, $(t,x)\mapsto f(t,J (\,\cdot\,))$ has a G\^ateaux derivative at every point $x\in X_\eta$: this is given by the linear operator on $X_\eta$
   $$
         \nabla_x (f(t,J x)) [h]= \nabla_u f(t,J x)[ J  h].
    $$ Finally, the function $(x,h) \to \nabla f(t,J x)[h]$
   is continuous as a map $X_\eta \times X_\eta \to \R$ and $\|\nabla_u f(t,J x)\|_{\eta}\leq C$, for $t\in [0,T],\, x\in X_\eta$ and a suitable constant $C>0$.
\end{rem}

Now we consider the dependence of the process $(\bx(t;s,x))_{t\geq 0}$ on the initial data. More precisely, we prove that $(\bx(t;s,x))_{t\geq 0}$ depends continuously on 
    $s$ and $x$ and it is also G\^ateaux
 differentiable with respect to $x$.      
The following result rely on Proposition $2.4$ in Fuhrman and Tessitore \cite{FuTe/2002}, where a parameter depending contraction principle is provided. 

\begin{prop}\label{prop:dif-state-eq}
    For any $p\geq 1$ the following holds. 
    \begin{enumerate}
        \item The map $(s,x) \mapsto \bx(t;s,x)$ defined on $[0,T]\times X_\eta$ and
        with values in $L^p(\Omega,C([0,T];X_\eta))$ is continuous.
        \item For every $s\in [0,T]$ the map $x \mapsto \bx(t;s,x)$ has, at every point
         $x\in X_\eta$, a G\^ateaux derivative $\nabla_x \bx(\cdot;s,x)$. The map
        $(s,x,h) \mapsto \nabla_x \bx(\cdot;s,x)[h]$ is a continuous map
        from $[0,T]\times X_\eta \times X_\eta \to L^p(\Omega,C([0,T];X_\eta))$ and, 
       for every $h\in X_\eta$, the following equation holds $\PP$-a.s.:
        \begin{multline*}
                \nabla_x \bx(t;s,x) [h]= e^{(t-s)B}h +\\ \int_s^t e^{(t-\tau)B} (I-B)P\, 
                 \nabla_u f(
                \tau, J \bx(\tau;s,x)) J \nabla_x \bx(\tau;s,x) [h]\dd \tau,  
        \end{multline*}
        for any $t\in [s,T]$, whereas $ \nabla_x \bx(t;s,x) [h]=h$ for $t\in [0,s]$.
    \end{enumerate}
\end{prop}

\begin{proof}
  \emph{Point $1$: continuity.} 
As before, we deal with the mapping $\Gamma,\Lambda$ and $\cK$ defined in 
 \eqref{eq:kappa}, \eqref{eq:lambda}, \eqref{eq:gamma} 
and we will denote $\Gamma$, $\Lambda$ and $\cK$ respectively by $\Gamma(\bx;s)$, $\Lambda(\cdot;s)$ and $\cK(\bx;s,x)$
in order to stress the dependence on the initial conditions $s$ and $x$.
Moreover
 we set $\cK(\bx;s,x)=x$, $\Gamma(\bx;s)=0$ and $\Lambda(\cdot\,;s)=0$ for $t<s$ and we recall that $\cK(\cdot;s,x)$ is a contraction,
with contraction constant independent on $s$ and $x$, in the space
 $L^p(\Omega,C([0,T];X_\eta))$ wih respect to the norm
\begin{align*}
     \tn \bx \tn^p_\eta:=\bE \sup_{t\in[0,T]} e^{-\beta t} \|\bx(t)\|_\eta^p.
\end{align*}

 By a parameter dependent contraction argument (see, for instance \cite[Proposition 2.4]{FuTe/2002}, the claim follows if we show that for all $\bx \in L^p(\Omega,C([0,T];X_\eta))$ the map
$t\mapsto \cK(\bx;s,x)$ is a continuous map from $[0,T] \times X_\eta $ with values in $L^p(\Omega,C([0,T];X_\eta))$.

To this end, we introduce two sequences $\left\{ s_n^+\right\}$ and $\left\{ s_n^-\right\}$ such that
   $s_n^+ \searrow s$ and $s_n^- \nearrow s$ and we estimate the norm of
 $\cK(\bx;s_n^+,x)-\cK(\bx;s_n^-,x)$ in the space $  L^p(\Omega,C([0,T];X_\eta))$. We have
\begin{multline*}
     \tn \cK(\bx;s_n^+,x)-\cK(\bx;s_n^-,x)\tn_\eta^p\leq 
     \sup_{t\in [0,T]} \| e^{(t-s_n^+)B}x-e^{(t-s_n^-)B}x\|_\eta^p \\
     +\sup_{t\in [0,T]} \| \Lambda(\bx;s_n^+)(t)-\Lambda(\bx;s_n^-)(t)\|_\eta^p 
      + \sup_{t\in [0,T]} \|\Gamma(t;s_n^+)-\Gamma(t;s_n^-)\|_\eta^p.
\end{multline*}
  
  Now we focus on the third member of the above inequality:  introducing a change of
variables we obtain
  \begin{align*}
       \bE \sup_{t\in [0,T]} \|\Gamma(t;s_n^+)-\Gamma(t;s_n^-)\|^p_\eta
       &\leq \bE \sup_{t\in [s_n^-,T]} \left\| \int_{s_n^-}^{t\wedge s_n^+} e^{(t-\tau)B}(I-B)P\,g
       \dd W(\tau) \right\|^p_\eta\\
       &\leq \bE \sup_{t\in [s_n^-,s_n^+]} 
      \left\| \Gamma(t;s_n^-)\right\|^p_\eta \\
     &\leq 
     \bE \sup_{t \in [0,s_n^+-s_n^-]} \left\|\Gamma(t;0)\right\|^p_\eta \to 0,
   \end{align*}
   where the final convergence comes as an immediate consequence of the dominated theorem,
since $\Gamma(\cdot\, ; 0 ) \in L^p(\Omega,C([0,T];X_\eta))$. 
Similarly,
\begin{align*}
    & \bE \sup_{t\in [0,T]} \left\|\Lambda(\bx;s_n^+)(t)-\Lambda(\bx;s_n^-)(t)\right\|^p_\eta\\
     & \qquad  \leq \bE \sup_{t\in [s_n^-,T]} \left\|\int_{s_n^-}^{t\wedge s_n^+} e^{(t-\tau)B} 
     (I-B) P f(\tau,J \bx(\tau)) \dd \tau \right\|^p_\eta \\
     & \qquad \leq \bE \sup_{t\in [s_n^-,s_n^+]} \left|\int_{s_n^-}^{t} \left\|e^{(t-\tau)B} 
     (I-B)^{\eta+1} P f(\tau,J \bx(\tau)) \right\|\dd \tau\right|^p \\
     & \qquad  \leq  C_{L,T}\, \bE \sup_{t\in [0,T]}  (1+\|\bx(\tau)\|_\eta^p)
       \sup_{t\in [s_n^-,s_n^+]} \left(\int_{s_n^-}^t (t-\tau)^{\theta-\eta-1} \dd \tau \right)^p\\
     &\qquad \leq C_{L,T} (s_n^+-s_n^-)^{p(\theta-\eta)} (1+\tn \bx \tn_\eta^p) \ \to 0.
\end{align*}
Finally, if we extend $e^{(t-s)B}$ to the identity for $t<s$ we have 
\begin{align*}
      \sup_{t\in [0,T]} \left\| e^{(t-s_n^-)B}x-e^{(t-s_n^-)B}x\right\|_\eta^p 
     = \sup_{t\in [s_n^-,T]} \left\| e^{(t-s_n^-)B}[x- e^{(s_n^+-s_n^-)B}x]\right\|_\eta 
      \ \to \  0
\end{align*}
and also the map $x \mapsto \left\{ t \mapsto e^{(t-s)B}x\right\}$ is clearly continuous 
    in $x$ uniformly in $s$ from $X_\eta$ into the space $C([0,T];X_\eta)$. 

\emph{Point $2$: differentiability.}
     Again by \cite[Proposition 2.4]{FuTe/2002}, it is enough to show that the map 
     $(\bx,s,x) \mapsto \cK(\bx;s,x)$ defined on $L^p(\Omega,C([0,T];X_\eta))\times [0,T]\times X_\eta$ with values in $L^p(\Omega,C([0,T];X_\eta))$ is G\^ateaux differentiable in 
$(\bx,x)$ and has strongly continuous derivatives. 

     The directional derivative $\nabla_\bx \cK(\bx;s,x)$ in the direction $\bh \in L^p(\Omega,C([0,H];X_\eta))$ is defined
    as 
     \begin{align*}
           \lim_{\ep \to 0} \frac{\cK(\bx + \ep \bh;s,x)-\cK(\bx;s,x)}{\ep}.
     \end{align*}
    We claim that the above limit coincides with the process 
     \begin{multline*}
        \nabla_\bx \cK(\bx;s,x)[\bh](t)= 
                \int_s^t e^{(t-\tau)B} (I-B)P\nabla_u f(\tau,J \bx(\tau))
                [J 
                \bh(\tau)] \dd \tau,  \quad t\in [s,T],
    \end{multline*}
    whereas $\nabla_\bx \cK(\bx;s,x)[\bh](t)=0$ when $t<s$. Moreover, the mappings
    $(\bx;s,x) \mapsto \nabla_\bx \cK(\bx;s,x)$ and $ \bh \mapsto \nabla_\bx \cK(\bx;s,x) [\bh]$ are continuous. In fact, let us define the process
     \begin{align*}
        I^\ep(t):&=  \frac{\cK(\bx + \ep \bh;s,x)(t)-\cK(\bx;s,x)(t)}{\ep} \\
          &\qquad - 
             \int_s^t e^{(t-\tau)B} (I-B)P 
            \nabla_u f(\tau,J \bx(\tau))
                [J 
                \bh(\tau)]  \dd \tau; 
     \end{align*}
    which, since we have the identity
    \begin{align*}
     & \int_s^t \int_0^1 
      e^{(t-\tau)B}(I-B)P   \nabla_u f(\tau,,J \bx(\tau) + \ep\, \xi\, J \bh(\tau))
                [J \bh(\tau)]\dd \xi \, \dd \tau\\
    &\quad \quad = \int_s^t \int_0^1 \frac{1}{\ep}\frac{\dd }{\dd \xi}  e^{(t-\tau)B}(I-B)P  f(\tau\, ,J \bx(\tau) + \ep\, \xi\, J \bh(\tau)), 
    \end{align*}
  can be rewritten as 
    \begin{align*}
          I^\ep(t) & =\int_s^t \dd \tau \left( \int_0^1  e^{(t-\tau)B}(I-B)P \nabla_u  f(\tau,J \bx(\tau) + \ep\, \xi\, J \bh(\tau)) [J 
               \bh(\tau)] \right.\\
          &\left.\qquad -  e^{(t-\tau)B}(I-B)P  \nabla_u  f(\tau,J \bx(\tau) )J 
                [\bh](\tau)\dd \xi \, \right).
     \end{align*}
%
   
      Moreover by the assumption on the gradient of $f$, for all $\ep >0$ we have
     \begin{align*}
      &\bE \sup_{t\in [s,T]} \left\|\int_s^t \dd \tau\int_0^1 e^{(t-\tau)B}(I-B)P \nabla_u f(\tau,J \bx(\tau) + \ep\, \xi\, J \bh(\tau)) 
                [\bh(\tau)] \dd \xi \right\|^p_\eta \\
      \leq\  &\|\nabla_u f\|^p \|P\|^p_{L(H;X_\theta)}\tn \bh\tn^p_\eta \sup_{t\in [s,T]} \left|\int_s^t \dd \tau \|e^{(t-\tau)B} (I-B) \|_{\mathcal{L}(X_\theta;X_\eta)}  \right|^p     \\
     \leq \  &  \|\nabla_u f\|^p\|P\|^p_{L(H;X_\theta)} \tn \bh \tn^p_\eta (T-s)^{p(\theta-\eta)} 
        < \infty
     \end{align*}
      and $e^{(t-\tau)B}(I-B)P \, \nabla_u f(\tau,J\bx)\, J \, $ is continuous in $\bx$. Therefore,
     by the dominated convergence theorem, we get 
      $ \tn I^\ep\tn_\eta^p \to 0$, as $\ep \to 0$ and the claim follows. 
      Continuity of the mappings $(\bx,s,x) \mapsto \nabla_\bx \cK (\bx;s,x)[\bh]$ and 
      $\bh \mapsto \nabla_\bx\cK(\bx;s,x)[\bh]$ can be proved in a similar way. 

    Finally, we consider the differentiability of $\cK(\bx;s,x)$ with respect to $x$. It is clear that the directional derivative $\nabla_x\cK(\bx;s,x)[h]$ in the direction $h\in X_\eta$ is the process given 
by 
\begin{align*}
     \nabla_x \cK(\bx;s,x)[h]= e^{(t-s)B}h, \qquad t\in [s,T],
\end{align*} 
 whereas $ \nabla_x \cK(\bx;s,x)[h]=h$ for $t\in [0,s]$ and the continuity of $\nabla_x \cK(\bx;s,x)[h]$ in all variables is immediate.   
       
\end{proof}

In the rest of the section we introduce an auxiliary process that will turn to be useful when 
dealing with the formulation of the fundamental relation  for the value function of our control problem. 
More precisely, for any $x\in X_\theta $ and $t\in [0,T]$ and $h$ in the space
\begin{align}\label{eq:D}
     \mathcal{D}:=\left\{ h\in X_{1-\theta} \subset X: \ (I-B)^{1-\theta}h \in X_\eta\right\},
\end{align}
we define the process  $(\Theta(t ;s,x)[h])_{t\in [0,T]}$ as
\begin{align}\label{eq:raptheta}
     \Theta(t; s,x)[h]:=( \nabla_x \bx(t; s,x) - e^{(t-s)B})(I-B)^{1-\theta} h, \qquad \textrm{for } t \in [s,T]
\end{align}
and $ \Theta(t; s,x)[h]:=0$ whenever $t\in [0,s)$. 
We notice that $\Theta(\cdot ; s,x)$ can be seen as a stochastic process with values into the space of linear mappings on $\mathcal{D}$. 
In the following we shall prove that it can be continuously extended to the whole space $X$ (that with an abuse of notation we still denote by $\Theta$). To this end, the first step is to verify
that the domain $\mathcal{D}$ is dense in $X$.
\begin{prop}
    The space $\mathcal{D}$ defined in\eqref{eq:D} is dense in $X$.
\end{prop}
\begin{proof}
     We recall that the linear operator $B$ is densely defined (see Theorem \eqref{t:state space setting}) and that $D(B)$ is contained in all real interpolation spaces 
   $X_\rho$, $\rho \in (0,1)$. 
   We notice that if $h\in D(B)$ we have
    \begin{align*}
                \|(I-B)^{1-\theta}h\|_\eta = \|(I-B)^{1-\theta+\eta}h\|_X < \infty.
   \end{align*}
    This implies that $D(B)\subset \mathcal{D}$ so that the claim follows. 
\end{proof}

Now we are ready to prove that $\Theta(t;s,x)$ can be extended as a linear operator 
from $X$ into itself.
\begin{prop}\label{prop:theta}
    There exists a process $\left\{\Theta(\cdot \ ;s,x)[h]: h\in X; x\in X_\eta, s\in [0,T]\right\}$ 
defined on $\Omega \times [0,T] \to X_\eta$ such that the following hold:
\begin{enumerate}
     \item The map $h\mapsto \Theta( \ \cdot \ ; s,x)[h]$ is linear and, on the space 
     $\mathcal{D}\subset X$, has the representation given in \eqref{eq:raptheta};
     \item The map $(s,x,h) \mapsto \Theta( \ \cdot \ ;s,x)[h]$ is continuous from $[0,T] \times 
       X_\eta \times X$ into $L^{\infty}(\Omega; C([0,T];X_\eta))$;
      \item There exists a positive constant $C$ such that 
     \begin{align}
        & |\Theta(\ \cdot \ ;s,x)[h]|_{L^{\infty}(\Omega; C([0,T];X_\eta))} \leq 
        C |h|_X    \label{eq:thetabound}, 
      \\& \qquad \qquad  \qquad \textrm{for all } s\in [0,T], x\in X_\eta, h\in X \notag.
     \end{align}
\end{enumerate}
\end{prop}

\begin{proof}
     For fixed $s\in [0,T]$, $x\in X_\eta$ and $h\in H$ we consider the integral equation 
    \begin{multline}\label{eq:theta}
       \Theta(t; s,x)[h]= \int_s^t e^{(t-\sigma)B}(I-B) P  \nabla_u f(\sigma, J  \bx(\sigma; 
        s,x)) J \Theta(\sigma;s,x)[h] \dd \sigma  \\
      + \int_s^t e^{(t-\sigma)B}(I-B) P \nabla_u f(\sigma, J  \bx(\sigma; 
        s,x)) J (I-B)^{1-\theta} e^{(\sigma - s)B}h\dd \sigma.
    \end{multline}
   Notice that 
   \begin{align*}
      &\int_s^t \left\| e^{(t-\sigma)B}(I-B)P \nabla_u f(\sigma, J  \bx(\sigma; 
        s,x)) J (I-B)^{1-\theta} e^{(\sigma - s)B}h \right\|_\eta \dd \sigma \\
        \leq &\ \int_s^t \| e^{(t-\sigma)B}(I-B)\|_{L(X_\theta;X)}\|P\|_{L(H;X_\theta)}\| \nabla_u f(\sigma, J \bx(\sigma; s,x))\|_{L(H)}\\
     & \qquad \qquad \qquad \| J\|_{\mathcal{L}(X_\eta;H)} \|(I-B)^{1-\theta}e^{(\sigma-s)B}\|_{L(X;X_\eta)}
  |h|_X \dd \sigma \\
    \leq & \ C |h|_X \int_s^t (t-\sigma)^{\theta-1} (\sigma-s)^{\theta-1-\eta} \dd \sigma \\
   \leq & \ C |h|_X (t-s)^{\theta-\eta+1}\int_0^1 (1-\sigma)^{\theta-1}  \sigma^{\theta-1-\eta} \dd \sigma\\
 = &\ C|h|_X B(\theta-\eta,\theta)
   \end{align*}
  where $B(\alpha,\beta)$ is the beta-distribution of parameter $\alpha,\beta>0$.
   By the above estimate we then obtain that equation \eqref{eq:theta} has $\PP$-almost surely 
  a unique mild solution in $C([s,T];X_\eta)$. Moreover, extending $\Theta(t;s,x)[h]=0$, for $t<s$, we have
    $\Theta(\cdot;t,x)[h] \in L^\infty(\Omega;C([0,T];X_\eta))$ and $|\Theta(\cdot;s,x)|_{L^{\infty}(\Omega;C([0,T];X_\eta))} \leq C |h|_X$. Concerning the continuity with respect to $s,x$ and $h$,
we can argue as in the proof of Proposition \ref{prop:dif-state-eq}.  Moreover, linearity is straight - forward. 

Finally, we prove the representation formula \eqref{eq:raptheta} for $\Theta(\cdot;s,x)[k]$ when $k\in \mathcal{D}$.  Setting $h=(I-B)^{1-\theta}k$, the equation satisfied by the G\^ateaux derivative of $\bx(\cdot;s,x)$ 
at $h$ is given by 
 \begin{multline*}
                (\nabla_x \bx(t;s,x) -e^{(t-s)B})(I-B)^{1-\theta}[k]  =\\ \int_s^t e^{(t-\tau)B} (I-B)P\, 
                 \nabla_u f(
                \tau, J \bx(\tau;s,x)) J \nabla_x \bx(\tau;s,x) (I-B)^{1-\theta}[k]  \dd \tau,  
        \end{multline*}
and by adding and subtracting the term $e^{(\tau-s)B} (I-B)^{1-\theta}[k] $ suitably
in the integral of the above inequality we obtain
 \begin{align*}
               & (\nabla_x \bx(t;s,x) -e^{(t-s)B})(I-B)^{1-\theta}[k]  \\
              & \quad = \int_s^t e^{(t-\tau)B} (I-B)P\, 
                 \nabla_u f(
                \tau, J \bx(\tau;s,x)) J e^{(\tau-s)B}(I-B)^{1-\theta}[k]  
                \dd \tau\\
            & \quad+\int_s^t e^{(t-\tau)B} (I-B)P\, 
                 \nabla_u f(
                \tau, J \bx(\tau;s,x)) J [\nabla_x \bx(\tau;s,x) -e^{(\tau-s)B}](I-B)^{1-\theta}[k]  \dd \tau.
        \end{align*}
Comparing the above equality with the equation satisfied by $\Theta(\cdot;s,x)[k]$ and applying Gronwall's lemma we conlcude that $\Theta(\cdot;s,x)[k]=(\nabla_x \bx(\tau;s,x) -e^{(\tau-s)B})(I-B)^{1-\theta}[k] $, $\PP$-a.s. for all $t\in [s,T]$. 

\end{proof}

\subsection{Regularity in the sense of Malliavin} \label{sec:Mallreg}
In order to state the main results concerning the Malliavin regularity of the process
$\bx$ we need to recall some basic definitions from
the Mlliavin calculus. We refer the reader the book \cite{N} for a detailed exposition.
The paper \cite{GroPa} treats the extension to Hilbert space valued random variables and processes. 
For every $h\in L^2(0,T;\Xi)$ we denote by $W(h)$ the integral
 $\int_0^T \langle h(t),\dd W(t)\rangle_{\Xi}$. Given a Hilbert space $K$, let us denote 
by $S_K$ the set of $K$-valued random variables $F$ of the form
$$
      F=\sum_{j=1}^m f_j(W(h_1),\dots,W(h_n))e_j,
$$
where $h_1,\dots,h_n\in L^2(0,T;\Xi)$, $\left\{ e_j\right\}$ is a basis of $K$ and
$f_1,\dots,f_m$ are infinitely differentiable functions $\R^n \to \R$ bounded together 
with all their derivatives. The Malliavin derivative $DF$ of $F\in S_K$ is defined as the process
$D_\sigma F$, $\sigma\in [0,T]$,
$$
     D_\sigma F=\sum_{j=1}^m \sum_{k=1}^n \partial_k f_j(W(h_1),\dots,W(h_n))e_j\otimes 
    h_k(s),
$$
with values in $L_2(\Xi,K)$; by $\partial_k$ we denote the partial derivatives with respect to the
$k$-th variable and by $e_j\otimes h_k(s)$ the operator $u\mapsto e_j \langle h_k(s),u\rangle_\Xi$. It is known that the operator $D: S_K\subset L^2(\Omega;K) \to L^2(\Omega \times [0,T]; L_2(\Xi;K))$ is closable. We denote by 
$\bD^{1,2}(K)$ the domain of its closure, and use the same letter to denote $D$ and its closure:
$$
     D: \ \bD^{1,2}(K) \subset L^2(\Omega;K) \ \to \ L^2(\Omega \times [0,T]; L_2(\Xi;K)).
$$ 
The adjoint operator of $D$,
$$
    \delta: \ {\rm dom}(\delta) \subset L^2(\Omega\times [0,T];L_2(\Xi;K)) \ \to \ L^2(\Omega;K),
$$
is called Skorohod integral. It is known that ${\rm dom}(\delta)$ contains 
$L^2_\cF(\Omega\times[0,T];L_2(\Xi;K))$ and the Skorohod integral of a process in this space
coincides with the It\^o integral; ${\rm dom}(\delta)$ also contains the class $\bL^{1,2}(L_2(\Xi;K))$, the latter being defined as the space of processes $u\in L^2(\Omega\times [0,T];L_2(\Xi;K))$ such that $u(t)\in \bD^{1,2}(L_2(\Xi;K))$ for a.e. $t\in [0,T]$ and there exists a measurable version $D_\sigma u(t)$ satisfying 
\begin{multline*}
     \| u\|^2_{\bL^{1,2}(L_2(\Xi;K))}= \|u\|^2_{L^2(\Omega\times [0,T];L_2(\Xi;K))}\\
     + \bE \int_0^T \int_0^T |D_\sigma u(t)|^2_{L_2(\Xi;L_2(\Xi;K))} \dd t \dd \sigma <\infty.
\end{multline*}
Moreover,$ \|\delta(u)\|^2_{L^2(\Omega;K)}\leq \|u\|^2_{\bL^{1,2}(L_2(\Xi;K))} $. 

The definition of $\bL^{1,2}(K)$ for an arbitrary Hilbert space $K$ is entirely analogous. 

We recall that if $F\in \bD^{1,2}(K)$ is $\cF_t$-adapted then $DF=0$ a.s. on $\Omega\times (t,T]$.

Now let us consider again the process $\bx=(\bx(t;s,x))_{t\in [s,T]}$, denoted simply
by $\bx(t)$, solution of the forward equation \eqref{eq:state-eq-2}, with $s\in [0,T],\  x\in X_\eta$ fixed. We set as before $\bx(t)=x$ for $t\in [0,s)$. We will soon prove that
 $\bx\in \bL^{1,2}(X_\eta)$. Then it is clear that the equality 
$D_\sigma \bx(t)=0$, holds, $\PP$-a.s., for a.a. $\sigma,s,t$ if $t <s$ or $\sigma>t$. 
\begin{prop}\label{prop:Mal}
    Assume Hypothesis \ref{hp:a,A,f,g,r,W}. Let $s\in [0,T]$ and $x\in X_\eta$ be fixed.
Then the following properties hold:
\begin{enumerate}
   \item \label{it:i-D} $\bx \in \mathbb{L}^{1,2}(X_\eta)$;
   \item \label{it:iii-D} 
   For a.a. $\sigma$ and $t$ such that $s\leq \sigma \leq t\leq T$, we have $\PP$-a.s. 
   \begin{multline}
   D_\sigma\bx(t)=(I-B) e^{(t-\sigma)B}P\ g + \\
      \int_\sigma^t e^{(t-r)B}(I-B) P \nabla_u f(r,J \bx(r))J D_\sigma \bx(r) \dd r \label{eq:eq-DMal}.
    \end{multline}
    and 
    \begin{align}
\|D_\sigma \bx(t)\|_{L_2(\Xi,X_\eta)} \leq C[(t-\sigma)^{\theta-\eta-1}+1], \label{eq:stima-l-inf}
     \end{align} from which it follows that for every $t\in [0,T]$ we have $D\bx(t) \in L^\infty(\Omega;L^2(0,T;L_2(\Xi,X_\eta)))$. 
\item \label{it:ii-D} There exists a version of $D \bx$ such that for every $ \sigma\in [0,T)$,
     $\left\{ D_\sigma \bx(t): \ t\in (\sigma,T]\right\}$ is a predictable process 
     in $L_2(\Xi;X_\eta)$ with continuous path, satisfying, for $p\in [2,\infty)$,
   \begin{align}\label{eq:ineq-it-ii}
        \sup_{\sigma \in [0,T]} \bE \left( \sup_{t\in (\sigma,T]}(t-\sigma)^{p(\theta-1-\eta)}\|D_\sigma \bx(t)\|^p_{L_2(\Xi,X_\eta)}
       \right) \leq C,
    \end{align}
    for some positive constant $C$ depending only on $p,L,T$ and $M:= \sup_{t\in [0,T]} |e^{tB}|$.
    Further, for every $t\in [0,T]$ we have $\bx(t)\in \bD^{1,2}(X_\eta)$. 
  
\item \label{it:iv-D} For every $q\in [2,\infty)$, the map $t\mapsto D\bx(t)$ is continuous from  $[0,T]$ (hence uniformly continuous and bounded) with values in $L^p(\Omega;L^2([0,T];L_2(\Xi;X_\eta))$. 
\end{enumerate}
\end{prop}

In order to prove this proposition we need some preparation. We start with the following lemma.
\begin{lem}\label{lem:DM}
    If $\bx \in \bL^{1,2}(X_\eta)$ then the random processes 
\begin{align*}
    \int_0^t e^{(t-r)B}(I-B)P \ f(r,J\bx(r)) \dd r \quad
\textrm{and} \quad
    \int_0^t e^{(t-r)}(I-B)P g \dd W(r)
\end{align*}
belong to $\bL^{1,2}(X_\eta)$ and for a.a. $\sigma$ and $t$ with $\sigma<t$,
\begin{align}
     &D_\sigma \int_0^t e^{(t-r)B}(I-B)P \ f(r,J \bx(r))\dd r =\notag\\
    &\qquad \qquad  \qquad \int_\sigma^t 
       (I-B) P \nabla_u f(r,J\bx(r)) J D_\sigma \bx(r) \dd r \label{eq:Df}
\intertext{while}
     &D_\sigma \int_0^t e^{(t-r)B}(I-B)P g \dd W(r) = e^{(t-r)B}(I-B)Pg \label{eq:DW}
\end{align}
\end{lem}

\begin{proof}
     We start by proving that the process
    $$
      I(t):=\int_0^t e^{(t-r)B}(I-B)Pf(r,J\bx(r))\dd r 
$$
belong to $\bL^{1,2}(X_\eta)$ and that equality \eqref{eq:Df} hold.
      By definition, we need to prove that for a.e. $t\in [0,T]$, the random variable
    $I(t)$ belong to $\bD^{1,2}(X_\eta)$
   with 
   \begin{align}\label{eq:stimaDf}
        \bE \int_0^T \int_0^T \left\| D_\sigma \int_0^t e^{(t-r)B}(I-B)Pf(r,J\bx(r))\dd r\right\|^2_{L_2(\Xi;H)}
    \dd r \,\dd \sigma <\infty.
\end{align}
 To prove \eqref{eq:Df} we will use the relation between the Malliavin derivative and the Skorohod integral. In particular, we recall that $\bD^{1,2}(X_\eta)={\rm dom}(\delta^\star)$ where $\delta^\star$ is 
   the adjoint of the Skorohod integral.  Hence, \eqref{eq:Df} will be proved once
we will have shown that, for any $\by \in L^2(\Omega \times [0,T];X_\eta)$ the equality
\begin{multline*}
      \langle \int_0^t e^{(t-r)B}(I-B) P f(r,J\bx(r)) \dd r , \delta(\by)\rangle_{L^2(\Omega;X_\eta)}  
   \\= \langle \int_0^t e^{(t-r)}(I-B)P D_{\cdot} f(r,J \bx(r)) \dd r, \by\rangle_{L^2(\Omega\times [0,T]; X_\eta)}
\end{multline*}
holds.
  For fixed $\by \in {\rm dom}(\delta)$ we consider the scalar product between $I(t)$ and 
   $\delta(\by)$ in the space $L^2(\Omega;X_\eta)$; applying Fubini's theorem we have
   \begin{align*}
        \langle I(t),\delta(\by)\rangle_{L^2(\Omega;X_\eta)} &= 
    \bE \langle  I(t) ,\delta(\by)  \rangle_\eta  \\
    &=  \int_0^t\bE \langle e^{(t-r)B}(I-B)Pf(r,J\bx(r)) ,\delta(\by)\rangle_\eta \ \dd r \\
    &= \int_0^t \langle  e^{(t-r)B}(I-B)Pf(r,J\bx(r)) ,\delta(\by) \rangle_{L^2(\Omega;X_\eta)} \ \dd r.
  \end{align*}
 Now using the duality between the operators $\delta$ and $D$ we obtain
   \begin{align*}
            & \langle I(t),\delta(\by)\rangle_{L^2(\Omega;X_\eta)} \\
           & \qquad = 
             \int_0^t \langle  D\left(  e^{(t-r)B}(I-B)Pf(r,J\bx(r))\right), \by\rangle_{L^2(\Omega \times [0,T];L_2(\Xi,X_\eta))}\\
       & \qquad =\int_0^t \int_0^t  \bE \langle D_\sigma \left( e^{(t-r)B} (I-B)P f(r,J\bx(r))\right), \by\rangle_{L_2(\Xi;X_\eta)} \dd \sigma \dd r \\
       &\qquad = \bE \int_0^t \langle\int_0^t e^{(t-r)B} (I-B)P D_\sigma  f(r,J\bx(r)) \dd r
        , \by\rangle \, \dd \sigma \\
      &\qquad = \langle \int_0^T e^{(t-r)B} (I-B)P D_{\cdot}  f(r,J\bx(r)) \dd r, \by\rangle_{L^2(\Omega \times [0,T];X_\eta)}
   \end{align*}
    Comparing the first and the last term in the above expression we conclude that
    $I(t)\in \bD^{1,2}(X_\eta)$ with
     $$
        \delta^\star I(t) = D I(t) = \int_0^t e^{(t-r)B} (I-B)P D_\cdot  f(r,J\bx(r)) \dd r.
    $$
   
 Now we prove
estimate \eqref{eq:stimaDf}.
   First we notice that 
\begin{align*}
    \bE  \int_0^T \|I(t)\|^2_{\eta} \dd t
     &\leq
    \bE \int_0^T \int_0^t \|e^{(t-r)B}(I-B)P f(r,J\bx(r))\|^2_{\eta} \dd r \ \dd t \\
   &\leq \ L^2\|P\|_{L(H;X_\theta)}^2\bE \int_0^T 
  \int_0^t (t-r)^{2(\theta-1-\eta)}(1+\|J\|_{L(X_\eta;H)}\|\bx(r)\|_\eta)^2 \dd r\ \dd t 
\end{align*}
The right-hand side is finite for a.a. $t\in [0,T]$; in fact, by exchanging the integrals 
we verify that 
\begin{multline*}
     \int_0^T \left( \bE\int_0^t (t-r)^{2(\theta-1-\eta)}(1+\|J\|_{L(X_\eta;H)}\|\bx(r)\|_\eta)^2     
  \dd r\right) \dd t\\
     \leq \int_0^T r^{-2(\theta-1-\eta)}\dd r  \int_0^T \bE(1+\|J\|_{L(X_\eta;H)}\|\bx(r)\|_\eta)^2 \dd r <\infty, 
\end{multline*}
since $\bx\in \bL^{1,2}(X_\eta) \subset L^2(\Omega \times [0,T];X_\eta)$. 
Next, 
 for every $t\in [0,T]$, by the chain rule for Malliavin derivative (see Fuhrman and Tessitore
\cite[Lemma 3.4 $(ii)$]{FuTe/2002}), $D_\sigma [f(r,J \bx(r))] = \nabla_u f(r,J \bx(r)) J D_\sigma 
\bx(r)$ for a.a. $\sigma>r$, whereas $D_\sigma [f(r,J \bx(r))]=0$ for a.a. $\sigma>r$, by adaptness.
Next, recalling the assumption on $\nabla_u f$,
\begin{align*}
   & \ \bE \int_0^T  \int_0^T \|D_\sigma I(t)\|^2_{L_2(\Xi,X_\eta)} \dd t \ \dd \sigma \\  \leq
   &\ \bE \int_0^T \int_0^T \int_\sigma^t \|e^{(t-r)B}(I-B)P\nabla_u f(r,J\bx(r))J D_\sigma \bx(r)\|^2_{L_2(\Xi;X_\eta)}
    \dd r \ \dd t \ \dd \sigma\\
     \leq & \ \|\nabla_u f\|_{L(H)}^2 \ \|J\|_{L(X_\eta;H)}^2\|P\|^2_{L(H;X_\theta)} \bE \int_0^T \int_0^\sigma \int_\sigma^t 
       (t-r)^{2(\theta-1-\eta)} \| D_\sigma \bx(r)\|^2_{L_2(\Xi;X_\eta)} \dd r\ \dd t \ \dd \sigma,
\end{align*} 
so that, by an easy application of Fubini's theorem,
\begin{align*}
    &\bE \int_0^T \int_0^T  \| D_\sigma I(t)\|^2_{L_2(\Xi;X_\eta)} \dd t \  \dd \sigma\\
    &\ \leq C \, \bE \int_0^T\int_0^t \int_\sigma^t   (t-r)^{2(\theta-1-\eta)} \| D_\sigma \bx(r)\|^2_{L_2(\Xi;X_\eta)} \dd r \ \dd \sigma \ \dd t\\
    &\ = C\, \bE \int_0^T \int_0^t \int_0^r 
     (t-r)^{2(\theta-1-\eta)} \|D_\sigma \bx(r)\|^2_{L_2(\Xi,X_\eta)} \dd \sigma\ \dd r \ \dd t\\
   &\ \ =C  \int_0^T \int_0^t (t-r)^{2(\theta-1-\eta)} 
    \int_0^r \bE \|D_\sigma \bx(r)\|^2_{L_2(\Xi;X_\eta)} \dd \sigma\ \dd r \ \dd t.
\end{align*}
Now the right hand side of the previous inequality is finite; in fact, 
by exchanging the integrals we verify that 
\begin{align*}
   & \int_0^T \left(\int_0^t (t-r)^{2(\theta-1-\eta)} \int_0^r  \bE \|D_\sigma \bx(r)\|^2_{L_2(\Xi;X_\eta)} \dd \sigma\ \dd r\right) \dd t \\
   &\quad \quad \leq \int_0^T r^{2(\theta-1-\eta)} \int_0^T \int_0^r \bE \| D_\sigma \bx(r)\|^2_{L_2(\Xi;X_\eta)} \dd \sigma \ \dd r\  \dd t\\
  &\quad \quad = \left(\int_0^T r^{2(\theta-1-\eta)} \dd r\right)\  \|D_\sigma \bx \|^2_{L^2(\Omega \times [0,T];L_2(\Xi;X_\eta))} <\infty,
\end{align*}
since $\bx\in \bL^{1,2}(X_\eta)$. This proves \eqref{eq:stimaDf}. 

Now we consider the Malliavin derivative of the stochastic term
$$
      \int_0^t e^{(t-r)B}(I-B)P g \dd W(r)
$$
and we prove that it belongs to $\bL^{1,2}(X_\eta)$ and satisfies \eqref{eq:DW}. 
 This will be consequence of an easy application of
the following fact, proved in \cite{FuTe/2002}, Proposition 3.4: if $\by\in \bL^{1,2}(X_\eta)$, and for a.a. $\sigma \in [0,T]$ the process $\left\{ D_\sigma \by(t): t\in [0,T]\right\}$ belongs to ${\rm dom}(\delta)$, and the map $\sigma\mapsto \delta(D_\sigma \bx)$ belongs to $L^2(\Omega\times [0,T];X_\eta)$, then $\delta (\by) \in \bD^{1,2}(X_\eta)$ and $D_\sigma \delta(\by)=\by(\sigma)+ \delta(D_\sigma \by)$.

We fix $t\in [0,T]$ and we define $\by(r):=e^{(t-r)B}(I-B)Pg$ for $t\geq r$ and $\by(r)=0, t<r$. Clearly $\by \in L^{1,2}(X_\eta)$ and $D_\sigma \by=0$ for every $\sigma\in [0,T]$. 
This implies that 
$(D_\sigma \by(t))_{t\geq 0} \in {\rm dom}(\delta)$ for a.a. $\sigma \in [0,T]$
and $\sigma \mapsto \delta(D_\sigma y)$ belongs to  $L^2(\Omega\times [0,T];X_\eta)$. 
On the other hand, 
we recall that the Skorohod and the It\^o integral coincide for adapted integrand, so that
\begin{align*}
     \delta (\by)= \int_0^t e^{(t-r)B}(I-B)Pg \dd W(r).
\end{align*}
Applying the result mentioned above we get $\delta(\by)\in \bL^{1,2}(X_\eta)$ with
\begin{align*}
     \delta(\by)&=\int_0^t e^{(t-r)B}(I-B)P g \dd W(r) \\
   &= e^{(t-r)B}(I-B)Pg 
\end{align*}
for $a.a. t\in [0,T]$ so that
formula \eqref{eq:DW} is proved. 


\end{proof}

Now for $\sigma \in [0,T)$ and for arbitrary predictable processes $(\bx(t))_{t\in [\sigma,T]}$ and $(\bq(t))_{t\in [\sigma,T]}$ with
values respectively in $X_\eta$ and $L_2(\Xi;X_\eta)$ we define the process $\cH=\cH(\bx,\bq)$
as  
\begin{align*}
    \cH(\bx,\bq )_{\sigma t} = e^{(t-\sigma)B}(I-B)P g + \int_\sigma^t e^{(t-r)B}(I-B)  P \nabla_u f(r, J \bx(r) ) J\bq(r) \dd r.
\end{align*}

We are now ready to prove Proposition \ref{prop:Mal}.
\begin{proof}[Proof of Proposition \ref{prop:Mal}]
     We fix $s\in [0,T)$. Let us consider the sequence 
  $\bx^{n}$ defined as follows: $\bx^0=0 \in X_\eta$,
   \begin{align*}
       \bx^{n+1}(t)&= e^{(t-s)B}x + \int_s^t e^{(t-r)B}(I-B)Pf(r,J\bx^n(r))\dd r + 
        \int_s^t e^{(t-r)B}(I-B)P g \dd W(r)\\
     & =\cK(\bx^n)(t),
   \end{align*}
and $\bx^n(t)=x$ for $t<s$,
where the mapping $\cK$ was defined in Section \ref{sec:state-eq} (see \eqref{eq:kappa}).
It was proved in Theorem \ref{thm:exuni} that $\cK$ is a contraction in $L^p_\cF(\Omega;C([0,T];X_\eta))$, hence, in particular, in the space $L^2(\Omega \times [0,T];X_\eta)$. This implies that $\bx^{n}$ converges to an element $\bx$ in this space. By Lemma \ref{lem:DM}, $\bx^{n+1}$ belongs to $\bL^{1,2}(X_\eta)$ and, for a.a. $\sigma$ and $t$ with $\sigma <t$
\begin{align}\label{eq:Dxn}
     D_\sigma \bx^{n+1} (t)= e^{(t-\sigma)B}(I-B)P g + \int_\sigma^t e^{(t-r)B}(I-B)P\nabla_u f(r,J \bx^n(r)) J D_\sigma \bx^n(r) \dd r.
\end{align}
   Hence, recalling the operator introduced above, we may write equality \eqref{eq:Dxn}
   as 
\begin{align*}
     D\bx^{n+1}=\cH(\bx^n,D\bx^n).
\end{align*}
In the following we prove that 
$$
       \tn\cH(\bx^n,D\bx^n) \tn_{\beta}^2 \leq \alpha\tn D\bx^n\tn_\beta^2,
$$
for some $\alpha \in [0,1)$ and $\beta>0$ to be chosen later, where $\tn \cdot \tn_\beta$
denotes an equivalent norm
 in $L^2(\Omega\times [0,T]\times  [0,T];L_2(\Xi;X_\eta))$. More precisely, for $\beta>0$, we introduce the norm 
$$
    \tn Y\tn_{\beta}^2:= \int_0^T \int_0^T e^{-\beta(t-\sigma)}\bE \|Y_{\sigma t}\|^2_{L_2(\Xi;X_\eta)}\dd t\, \dd \sigma.
$$
First we estimate the 
term $U:=\left\{U_{\sigma t}: 0\leq \sigma \leq t \leq T\right\}$ defined as $U_{\sigma t}:=e^{(t-\sigma)B}(I-B)Pg$. We have
\begin{align*}
       \tn U\tn_\beta^2 &= \int_0^T \int_0^T e^{-\beta(t-\sigma)}
       \| e^{(t-\sigma)B}(I-B)P g \|^2_{L_2(\Xi;X_\eta)} \dd t \ \dd \sigma \\
     & \leq C_{P,g}\int_0^T \int_0^t e^{-\beta(t-\sigma)} (t-\sigma)^{2(\theta-1-\eta)}\dd \sigma \ \dd t \\
   & = C_{P,g} \int_0^T \int_0^t e^{-\beta\sigma} \sigma^{2(\theta-1-\eta)} \dd \sigma \ \dd t\\
   & \leq C_{P,g} \,T \,\frac{ \Gamma(2(\theta-1-\eta)) }{\beta^{2(\theta-1-\eta)+1}}<\infty,
\end{align*}
where $C_{P,g}$ is a positive constant depending only on the norms of $P$ and $g$. 
Hence $\left\{e^{(t-\sigma)B}(I-B)Pg: \ 0\leq \sigma<t\leq T\right\}$ is bounded in the space
$L^2(\Omega\times [0,T] \times [0,T];L_2(\Xi;X_\eta)) $. 
Now let us consider the norm of $\cH(\bx^n,D\bx^n)$: taking into account the above inequality we get
   \begin{multline*}
      \tn \cH(\bx^n,D\bx^n) \tn^2_\beta=\int_0^T \int_0^T e^{-\beta(t-\sigma)}\bE \| \cH(\bx^n,D\bx^n)_{\sigma t}\|_{L_2(\Xi;X_\eta)}^2
    \dd \sigma \ \dd t  \\
    \leq C_{P,g}\nu(\beta)+ \int_0^T \int_0^t\int_\sigma^t e^{-\beta(t-\sigma)}\bE\left\| e^{(t-r)B}(I-B)P \nabla_u f(r,J \bx^b(r))J  D_\sigma \bx^n(r) \right\|^2_\eta \dd r \ \dd \sigma \ \dd t
\end{multline*} 
 we set $\nu(\beta):=  T \,\frac{ \Gamma(2(\theta-1-\eta)) }{\beta^{2(\theta-1-\eta)+1}}$.
Now changing the order of integration, we have
\begin{align*}
      \tn \cH(\bx^n,D\bx^n) \tn^2_\beta &\leq C_{P,g}  \nu(\beta)+ 
     C_{P,f,J}\int_0^T \int_\sigma^T\left(  \int_r^T (t-r)^{-2(\theta-1-\eta)}e^{-\beta(t-r)} \dd t 
     \right) \\
  & 
    \qquad \qquad \qquad \qquad \quad e^{-\beta(r-\sigma)}\bE \|D_\sigma \bx^n(r)\|^2_{L_2(\Xi;X_\eta)}\, \dd r \,  \dd \sigma\\
    &  \leq C_{P,g}\nu(\beta)
   +   C_{P,f,J}\int_0^T \int_\sigma^T 
      \left( \sup_{r\in [\sigma,T]}\int_r^T (t-r)^{-2(\theta-1-\eta)}e^{-\beta(t-r)} \dd t \right)    \\
   &   \qquad \qquad \qquad \qquad \quad
    e^{-\beta(r-\sigma)}\bE \|D_\sigma \bx^n(r)\|^2_{L_2(\Xi;X_\eta)}\dd r\, \dd \sigma,
\end{align*}
where $C_{P,f,J}$ is a positive constant depending only on the norms of $P,f,J$.
   Since the supremum on the right-hand side can be estimated by $\int_0^T t^{2(\theta-1-\eta)}
e^{-\beta t} \dd t$ we obtain 
\begin{align*}
      & \tn \cH(\bx^n,D\bx^n) \tn^2_\beta
     \leq C_{P,g}\nu(\beta)+  C_{P,f,J} \nu(\beta)\\ 
     & \qquad \qquad \qquad \qquad 
     \int_0^T \int_\sigma^T e^{-\beta(r-\sigma)}\bE \|D_\sigma \bx^n(r)\|^2_{L_2(\Xi;X_\eta)}\dd r \ \dd  \sigma.
\end{align*}
Now we choose $\beta$ large enough such that $\nu(\beta)(1+\tn D\bx^n\tn_\beta^2 ) \leq \alpha 
\tn D\bx^n\tn_\beta^2$, for $\alpha \in [0,1)$, so that the above inequality means that
\begin{align}\label{eq:cH}
     \tn\cH(\bx^n,D\bx^n)\tn_\beta^2 \leq \alpha \tn D\bx^n\tn_\beta^2,
\end{align}
with $\alpha\in [0,1)$.
From \eqref{eq:cH} and from the fact that $e^{(t-\sigma)B}(I-B)Pg$ is bounded in $L^2(\Omega\times [0,T]\times[0,T]; L_2(\Xi;X_\eta))$, it follows that 
  the sequence $\left\{D \bx^n\right\}_{n\in \bN}$ is also bounded in this space. Since, as mentioned before 
$\bx^n$ converges to $\bx$ in $L^2(\Omega\times [0,T];X_\eta)$ it follows from the 
closedness of the operator $D$ that $\bx$ belongs to $\bL^{1,2}(X_\eta)$. Thus point \ref{it:i-D}
of Proposition \ref{prop:Mal} is proved.

We now consider point \ref{it:iii-D} in Proposition \ref{prop:Mal}.
First, we notice that, by Lemma \ref{lem:DM}, we can compute the Malliavin derivative of both side of equation \eqref{eq:state-eq-2} and we obtain, for a.a. $\sigma$ and $t$ such that $\sigma <t$
the following equality $\PP$-a.s.:
\begin{multline}\label{eq:Dx}
    D_\sigma \bx(t)= e^{(t-\sigma)B}(I-B)Pg \\+ \int_\sigma^t e^{(t-r)B}(I-B)P \nabla_u f(r,J \bx(r))
    J D_\sigma \bx(r) \dd r.
\end{multline}

Since, for any $\sigma <t$
\begin{align}\label{eq:t-iniz}
     \|e^{(t-\sigma)B}(I-B)Pg\|_{L_2(\Xi;X_\eta)}  \leq (t-\sigma)^{(\theta-1-\eta)}
         \|P\|_{L(H;X_\theta)} \|g\|_{L_2(\Xi;X_\eta)},
\end{align}
by the boundedness of $\nabla_u f$ and the Gronwall's lemma it easy to deduce that 
$$
      \|D_\sigma \bx(t)\|_{L_2(\Xi,X_\eta)} \leq C ((t-\sigma)^{\theta-\eta-1}+1).
$$
In particular it follows that, for every $ t\in [0,T]$, the mapping $\sigma \mapsto D_\sigma\bx(t)$ is bounded in $L^\infty(\Omega,L^2(0,T;L_2(\Xi;X_\eta)))$. 

We now consider point \ref{it:iii-D} in Proposition \ref{prop:Mal}.
which concerned with
the regularity of the trajectories of $D_\sigma\bx$, $\sigma \in [0,T)$. 
We introduce the space $\cV$ of processes $(\bq_{\sigma t})_{ 0\leq \sigma \leq t \leq T}$, such that, for every $\sigma \in [s,T)$, 
$(\bq_{\sigma t})_{ t\in (\sigma,T]}$ is a predictable process in $X_\eta$ with continuous paths, and such that 
\begin{align*}
     \sup_{\sigma \in [0,T] }\bE\left(\sup_{t\in (s,T]}e^{-\beta(t-\sigma)p}(t-\sigma)^{p(\theta-1-\eta)}\|\bq_{\sigma t}\|^2_{L_2(\Xi;X_\eta)} \right) <\infty.
\end{align*}
Here $p\in [2,\infty)$ is fixed and $\beta>0$ is a parameter to be chosen later. 
Let us consider the equation
\begin{align}\label{eq:q}
     \bq_{\sigma t}= e^{(t-\sigma)B}(I-B)Pg + \int_\sigma^t e^{(t-r)B}(I-B)P\nabla_u f(r,J\bx(r)) J\bq_{\sigma r}\dd r,
\end{align}
which we can rewrite as
\begin{align*}
    \bq_{\sigma t}=\cH(\bx,\bq)_{\sigma t}, \qquad  t\in (\sigma,T].
\end{align*}
We claim that equation
\eqref{eq:q} admits a unique mild solution in the space $\cV$. To prove this, it is suffices to show that the term $(U_{\sigma t})_{0\leq \sigma \leq t \leq T}$ defined before
belongs to $\cV$ and that $\cH$ is a contraction in the space
$\cV$. Since, for any $\sigma <t$
$$
     \|e^{(t-\sigma)B}(I-B)Pg\|_{L_2(\Xi;X_\eta)}  \leq (t-\sigma)^{(\theta-1-\eta)}
         \|P\|_{L(H;X_\theta)} \|g\|_{L_2(\Xi;X_\eta)}
$$ 
we have
$$
    \sup_{\sigma \in [0,T]} \sup_{t\in (\sigma,T]} (t-\sigma)^{p(1+\eta-\theta)}\|e^{(t-\sigma)B}(I-B)Pg\|_{L_2(\Xi;X_\eta)}   <\infty.
$$
This proves that $U\in \cV$. 
Further, repeating almost identical passages as those leading to \eqref{eq:Dxn},
we can prove that $\cH(\bx,\cdot)$ is a contraction in the space $\cV$ provided that $\beta$ is chosen 
sufficiently large. This proves the claim.

Now we prove that $\bq$ is the required version of $D\bx$.
 Subtracting \eqref{eq:q} from \eqref{eq:Dx}, we obtain, $\PP$-a.s., for a.a. $\sigma$ and $t$ with $\sigma <t$
\begin{align*}
     D_\sigma \bx(t)-\bq_{\sigma t} = \int_\sigma^t e^{(t-r)B} (I-B)P \nabla_u f(r,J \bx(r))
    \left[ J (
       D_\sigma \bx(r)-\bq_{\sigma r})\right] \dd r.
\end{align*}
Repeating the passages that led to \eqref{eq:cH}, we obtain
\begin{align*}
    \tn D_\sigma \bx -\bq\tn_\beta^2 \leq \alpha \tn D_\sigma \bx-\bq\tn^2_{\beta}
\end{align*}
for some $\alpha<1$, which clearly implies that $\bq$ is a version of $D\bx$. 

To conclude the proof of \ref{it:ii-D} in Proposition
\ref{prop:Mal} we need to prove that $\bx(t)\in \bD^{1,2}(X_\eta)$ for every $t\in [0,T]$. This
assertion is clear for $t\in [0,s]$ since $\bx(t)=x$ for $t<s$. For $t\in (s,T]$ we take a sequence $t_n \searrow t$ such that $\bx(t_n) \in \bD^{1,2}(X_\eta)$ and we note that 
by \eqref{eq:ineq-it-ii}, the sequence 
$$
 \bE\int_0^T \|D_\sigma \bx(t_n)\|^2_{L_2(\Xi;X_\eta)}\dd \sigma 
$$
is bounded by a constant independent of $n$. Since $\bx(t_n) \to \bx(t)$ in $L^2(\Omega; X_\eta)$, it follows from the closedness of the operator $D$ that $\bx\in \bD^{1,2}(X_\eta)$. 


Now we prove point \ref{it:iv-D} in Proposition \ref{prop:Mal}. To prove that $t\mapsto D\bx(t)$ is continuous as a mapping from $[0,T]$ into $L^p(\Omega;L^2(0,T;X_\eta))$ we fix $t\in [0,T]$
and let $t_n^+ \searrow t$ and $t^n_- \nearrow t$. Then, since $D\bx$ satisfies 
equation \eqref{eq:q}, we have
\begin{multline}\label{eq:ineq-cont}
   \bE \int_0^T \|D_\sigma \bx(t_n^+)-D_\sigma \bx(t_n^-)\|^2_\eta \dd \sigma  
   \leq 2\int_0^T \left\|(e^{(t_n^+-\sigma)B}-e^{(t_n^- - \sigma)B})(I-B)Pg\right\|_\eta^2 \dd \sigma\\
     +   2\bE \int_0^T \left\| \int_\sigma^{t_n^+} e^{(t_n^+-r)B}(I-B)P \nabla_u 
      f(r, J\bx(r)) J D_\sigma\bx(r) \dd r\right. \\
      -   \left.  \int_\sigma^{t_n^-} e^{(t_n^- -r)B}(I-B)P \nabla_u 
      f(r, J\bx(r)) J D_\sigma\bx(r) \dd r\right\|^2_\eta
\end{multline}
We estimate the two terms in the right member of the above inequality separately. Concerning the 
first one, which is deterministict we have
\begin{align*}
   &  2\int_0^T \left\|(e^{(t_n^+-\sigma)B}-e^{(t_n^- - \sigma)B})(I-B)Pg\right\|_\eta^2 \dd \sigma
    \\
   = \ & \int_{t_n^-}^{t_n^+} \left\|e^{(t_n^+ - \sigma)B}(I-B)Pg \right\|^2_\eta \dd \sigma
         + \int_{0}^{s_n^-} \left\| (e^{(t_n^+-t_n^-)B}-I)e^{(t_n^- - \sigma)B}(I-B)Pg\right\|^2_\eta\dd \sigma \\
     \leq \ &  C_{P,g} \int_{t_n^-}^{t_n^+} (t_n^+-\sigma)^{2(\theta-\eta-1)} \dd \sigma
     + C_{P,g} \|e^{(t_n^+-t_n^-)B}-I\|^2\int_0^{t_n^-} (t_n^--\sigma)^{2(\theta-1-\eta)} \dd \sigma 
\end{align*}
We notice that both integrals in the last part of the above expression go to $0$ as $n\to \infty$. 

In a similar way, using the bound \eqref{eq:ineq-it-ii} on $D_\sigma\bx$ for $p=\infty$, we conclude that also the second integral in \eqref{eq:ineq-cont} goes to $0$ and the required continuity follows.

\end{proof}

Now we denote by $\left\{ e_j\right\}_{j\in \bN}$ a basis on the space $\Xi$ and consider 
the standard real Wiener process $W^j(\tau)=\int_T \langle e_j,W(s)\rangle\dd s$.   
We conclude the section by investigating the existence of the joint quadratic variation of $W^j, i\in \bN$ with a process of the form $\left\{w(t,\bx(t)): \ t\in [0,T]\right\}$ for a given function $w: [0,T]\times X \to \R$, on an interval $[0,s] \subset [0,T)$. As usual this is defined as the limit in probability of
$$
      \sum_{i=1}^n (w(t_i,\bx(t_i))-w(t_{i-1},\bx(t_{i-1})))(W^j(t_i)-W^j(t_{i-1}),
$$
where $\left\{ t_i: 0=t_0<t_1 <\dots <t_n,\ i=0,\dots,n\right\}$ is an arbitrary subdivision of $[0,t]$ whose mesh tends to $0$. We do not require that convergence takes place uniformly in time. This definition is easily adapted to arbitrary interval of the form $[s,t] \subset [0,T)$.
Existence of the joint quadratic variation is not trivial. Indeed, due to the occurrence of convolution type integrals in the definition of mild solution, it is not obvious that the process $\bx$ i a semimartingale. Moreover, even in this case, the process  $w(\cdot, \bx)$ might fail to be a semimartingale if $w$ is not twice differentiable, since It\^o formula does not apply. Nevertheless, the following result hold true. Its proof could be deduced from generalization of some result obtained in \cite[pg. 193]{DeFuTe} to the infinite-dimensional case, but we prefer to give a simpler direct proof. 
\begin{prop}\label{prop:quad-var}
Suppose that $w\in C([0,T)\times X_\eta;\R)$ is G\^ateaux differentiable with respect to $\bx$,
and that for every $s<T$ there exist constant $K$ and $m$ (possibly depending on $s$)
such that 
\begin{align}\label{eq:stima-w}
     |w(t,x)|\leq K(1+|x|)^m, \quad |\nabla w(t,x)|\leq K(1+|x|)^m, \quad t\in [0,s], x\in X.
\end{align}
Let $\eta$ and $\theta$ satisfy condition \eqref{eq:eta-theta} in Theorem \ref{t:state space setting}. 
Assume that for every $t\in [0,T)$, $x\in X_\eta$ the linear operator $k\mapsto \nabla w(t,x)(I-B)^{1-\theta}k$ (a priori defined for $k\in \mathcal{D}$) has an extension to a bounded linear operator $X \to \R$, that we denote by $[\nabla w(I-B)^{1-\theta}](t,x)$. Moreover, assume that the map
$(t,x,k) \mapsto [\nabla w(I-B)^{1-\theta}](t,x)k$ is continuous from $[0,T) \times X_\eta \times X $ into $\R$. 
For $t\in [0,T),x\in X_\eta$, let $\left\{ \bx(t;s,x),t\in[s,T]\right\}$ be the solution of equation 
\eqref{eq:un-state-eq}. Then the process $\left\{w(t,\bx(t;s,x)), t\in [s,T]\right\}$ admits a joint quadratic variation process with $W^j$, for every $j\in \bN$, on every interval $[s,t] \subset [s,T)$, given by
\begin{align*}
     \int_s^t [\nabla w(I-B)^{1-\theta}](r,\bx(r;s,x))(I-B)^{\theta} P g e_j\dd r.
\end{align*}
\end{prop}
\begin{proof}
     For semplicity we write the proof for the case $s=0$ and we write $\bx(t)=\bx(t;s,x)$,
    $w(t)=w(t,\bx(t))$. It follows from the assumptions that the mapping $(t,x,h) \mapsto 
     \nabla w(t,x)h$ is continuous on $[0,T)\times X_\eta \times X$. By the chain rule for the Malliavin derivative operator (see Fuhrman and Tessitore \cite{FuTe/2002}), it follows that for every $t<T$
 we have $w(t)\in \bD^{1,2}(X_\eta)$ and $Dw(t,\bx(t))D\bx(t)$. 

Let us now compute the joint quadratic variation of $w$ and $W^j$ on a fixed interval $[0,t]\subset [0,T)$. Let $0=t_0 <t_1 \dots <t_n=t$ be a subdivision of $[0,t] \subset [0,T]$ with mesh $\delta={\rm max}_{i }(t_i-t_{i-1})$. By well-known rules of Malliavin calculus ({\bf see}) we have
\begin{multline*}
    (w(t_i)-w(t_{i-1}))(W^j(t_i)-W^j(t_{i-1})= (w(t_i)-w(t_{i-1})\int_{t_{i-1}}^{t_i} \dd W^j(t)\\
      =\int_{t_{i-1}}^{t_i} D_\sigma(w(t_i)-w(t_{i-1})) \dd \sigma + \int_{t_{i-1}}^{t_i} (w(t_i)-w(t_{i-1}))\hat{\dd} W^j(\sigma), 
\end{multline*} 
where we use the symbol $\hat{\dd }W^j$ to denote the Skorohod integral. We note that $D_\sigma w(t_{i-1})=0$ for $\sigma>t_{i-1}$. Therefore setting 
$$U_\delta(\sigma)=\sum_{i=1}^n
   (w(t_i)-w(t_{i-1})){\bf 1}_{(t_{i-1},t_i]}(\sigma)$$ we obtain
    \begin{multline*}
       \sum_{i=1}^n(w(t_i)-w(t_{i-1}))(W^j(t_i)-W^j(t_{i-1})\\
     = \int_0^t U_\delta(\sigma) \hat{\dd }W^j(\sigma) + \sum_{i=1}^n \int_{t_{i-1}}^{t_i} \nabla w(t_i,\bx(t_i))D_\sigma \bx(t_i)e_j\dd \sigma.
   \end{multline*}
  Recalling the equation satisfied by $D\bx$ (see \eqref{eq:Dx}) we have 
\begin{equation}\label{eq:var-quad-w}
   \begin{aligned}
     &\sum_{i=1}^n (w(t_i)-w(t_{i-1}))(W^j(t_i)-W^j(t_{i-1}))\\
    &\quad = \int_0^t U_\delta(\sigma) \hat{\dd }W^j(\sigma) \\
       &\quad + \sum_{i=1}^n  \int_{t_{i-1}}^{t_i} \nabla w(t_i,\bx(t_i)) e^{(t_i-\sigma)B}(I-B)P g e_j
         \dd \sigma\\
     & \quad+\sum_{i=1}^n  \int_{t_{i-1}}^{t_i} \nabla w(t_i,\bx(t_i)) \int_\sigma^{t_i} e^{(r-\sigma)B}(I-B)P
     \nabla_u f(r,J\bx(r))J D_\sigma \bx(r)e_j \dd r \ \dd \sigma.
\end{aligned}
\end{equation} 
Now we let the mesh $\delta$ go to $0$. We discuss the three terms in the right-member of the above inequality separately. 

Using the continuity properties of the maps $t\mapsto \bx(t)$ and $t\mapsto D\bx(t)$ stated in Proposition \ref{prop:Mal} and taking into account the continuity properties of $w$ and $\nabla w$, the estimate \eqref{eq:stima-w} and the chain rule $D_\sigma w(t)=w(t,\bx(t))D_\sigma \bx(t)$,
it is easy to see that the map $t\mapsto w(t)=w(t,\bx(t))$ is continuous from $[0,T]$ to 
$\bD^{1,2}(X_\eta)$. It follows that $U_\delta \to 0$ in $\bL^{1,2}(X_\eta)$ and by the continuity 
of the Skorohod integral we conclude that the first term in the right-member of \eqref{eq:var-quad-w} goes to $0$ as $\delta \to 0$.

According to the definition of $\nabla w(t,\bx(t))$ the second term can be written as
\begin{multline*}
    \sum_{i=1}^n [\nabla w(I-B)^{1-\theta}](t_i,\bx(t_i))\int_{t_{i-1}}^{t_i}
     e^{(t_i-\sigma)B}(I-B)^\theta Pg e_j\dd \sigma \\
   \leq  \sum_{i=1}^n [\nabla w (I-B)^{1-\theta}](t_i,\bx(t_{i}))(I-B)^{\theta}Pg e_j(t_i-t_{i-1})
     \\
      +  \sum_{i=1}^n [\nabla w(I-B)^{1-\theta}](t_i,\bx(t_i))\int_{t_{i-1}}^{t_i}
    ( e^{(t_i-\sigma)B}-I)(I-B)^{\theta}Pg e_j\dd \sigma.
\end{multline*}
We notice that, for every $i=0,\dots,n$,
\begin{align*}
     \sup_{\sigma \in [t_{i-1},t_i]}\left| ( e^{(t_i-\sigma)B}-I)(I-B)^{\theta}Pg e_j\right| 
     \leq \sup_{\sigma \in [0,\delta]} \left|  ( e^{\sigma B}-I)(I-B)^{\theta}Pge_j\right| \to 0,
\end{align*}
as $\delta \to 0$, by the strong continuity of the semigroup. From the properties of $ [\nabla w (I-B)^{1-\theta}]$ and the continuity of the paths of $\bx$ it follows that 
\begin{multline*}
      \sum_{i=1}^n [\nabla w(I-B)^{1-\theta}](t_i,\bx(t_i))\int_{t_{i-1}}^{t_i}
     e^{(t_i-\sigma)B}(I-B)^\theta Pge_j \dd \sigma \\ \to 
     \ \int_0^t [\nabla w(I-B)^{1-\theta}] (r,\bx(r))(I-B)^{\theta}Pge_j\dd r, \quad \PP-a.s.
\end{multline*}
Now recalling that, by Proposition \ref{prop:Mal} - \ref{it:iii-D} $D_\sigma \bx \in L^\infty(\Omega;L^2(0,T;X_\eta))$  and the boundedness of $\nabla f$ we can estimate the third term in the right-member of \eqref{eq:var-quad-w} by
\begin{multline*}
       \left|\sum_{i=1}^n  \int_{t_{i-1}}^{t_i} \nabla w(t_i,\bx(t_i)) \int_\sigma^{t_i} e^{(r-\sigma)B}(I-B)P
     \nabla_u f(r,J\bx(r))J D_\sigma \bx(r) e_j\dd r \ \dd \sigma \right| \\
    \leq C_{P,f}  \sum_{i=1}^n \|\nabla w(t_i,\bx(t_i))\|_{L(X_\eta)} \int_{t_{i-1}}^{t_i} \int_\sigma^{t_i} ((r-\sigma)^{\theta-1-\eta}+1)\dd r \ \dd \sigma \\
    \leq C_{P,f,\theta,\eta}  \sum_{i=1}^n \|\nabla w(t_i,\bx(t_i))\|_{L(X_\eta)} [(t_{i}-t_{i-1})\delta^{\theta-\eta}+(t_i-t_{i-1})\delta]
\end{multline*}
where $C$ is a positive constant depending only on $P$ and $f$.
But the last term goes to $0$, $\PP$-a.s., by the continuity properties of $\nabla w$ and the continuity of the paths of $\bx$.

\end{proof}

\section{The backward stochastic differential equation}\label{sec:back}
In this section we consider the backward stochastic differential equation in the unknown
 $(Y,Z)$:
\begin{align}\label{eq:bsde1}
  \begin{cases}
      \dd Y(\tau)= \psi(\tau,\bx(\tau;t,x),Z(\tau))\dd \tau+ Z(\tau) \dd W(\tau), \qquad \tau\in [s,T],\\
      Y(T) = \phi(\bx(T;s,x))
   \end{cases}
\end{align}
where $\bx( \cdot ;s,x)$ is the solution of the uncontrolled equation \eqref{eq:un-state-eq}
(with the convention $\bx(\tau;s,x)=x$ for $\tau \in[0,s)$) and  $\psi$ is the Hamiltonian function relative to the control problem described in Section \ref{sec:intr}. More precisely, for $ t\in [0,T], x\in X_\eta, z\in \Xi^*$ we have
\begin{align}\label{eq:Ham}
     \psi(t,x,z)=\inf\left\{l(t,Jx,\gamma)+z \ r(t,Jx,\gamma): \ \gamma \in \cU \right\}. 
\end{align}

For further use, we prove some additional properties of the function $\psi$:
\begin{prop}\label{prop:psi}
     Under Hypothesis \ref{hp:a,A,f,g,r,W} and \ref{hp:l,phi} the following holds:
    \begin{enumerate}
       \item \label{it:i} There exists a positive constant $C$ such that for any $t\in [0,T], x_1,x_2 \in X_\eta$ and $ \ z_1,z_2\in \Xi^*, $
    $$|\psi(t,x_1,z_1)-\psi(t,x_2,z_2)| \leq
       C|z_1-z_2|+C (1+\|x_1\|_\eta+\|x_2\|_\eta) \|x_1-x_2\|_\eta; $$
    \item \label{it:ii*} Thee exists a positive constant $C$ such that 
    $$
          \sup_{t\in [0,T]}|\psi(t,0,0)|\leq C.
    $$
    \end{enumerate}
\end{prop}
\begin{proof}
   Point \ref{it:ii*} follows from the fact that
   $$
         |\psi(t,0,0)|=\left| \inf_{\gamma \in \cU}l(t,0,\gamma)\right|. 
   $$
   Moreover, for all $\gamma \in \cU$ we have
   \begin{align*}
      l(t,Jx,\gamma)+zr(t,Jx,\gamma) & \leq l(t,J x^\prime, \gamma)+z^\prime r(t,Jx^\prime,\gamma) +|l(t,Jx,\gamma)-l(t,Jx^\prime,\gamma)| \\
       & \ + |zr(t,Jx,\gamma)-z^\prime r(t,Jx^\prime,\gamma)| \\
      &\leq  l(t,J x^\prime, \gamma)+ z^\prime r(t,Jx^\prime,\gamma) +|l(t,Jx,\gamma)-l(t,Jx^\prime,\gamma)|\\
      & \ + |(z-z^\prime)r(t,Jx,\gamma)|+|z^\prime(r(t,Jx,\gamma)-r(t,Jx^\prime,\gamma)|\\
    & \leq  l(t,J x^\prime, \gamma)+ z^\prime r(t,Jx^\prime,\gamma) + C \|J\|\|x-x^\prime\|_\eta(1+\|x\|_\eta+\|x^\prime\|_\eta)\\
  & \ +  C |z-z^\prime|+C \|J\|\|x-x^\prime\|_\eta(1+\|x\|_\eta+\|x^\prime\|_\eta)|z^\prime|,
   \end{align*}
   and taking the infimum over $\gamma$,
   \begin{align*}
     \psi(t,x,z)&\leq\psi(t,x^\prime,z^\prime)+C \|J\||x-x^\prime|(1+\|x\|_\eta+\|x^\prime\|_\eta)+
   C |z-z^\prime|\\
   &\ +C \|J\||x-x^\prime|(1+\|x\|_\eta+\|x^\prime\|_\eta)|z^\prime|\\
   & \leq c|z-z^\prime|+c|x-x^\prime|(1+|z|+|z^\prime|)(1+\|x\|_\eta+\|x^\prime\|_\eta)
   \end{align*}
   for some $c$. Exchanging $x,z$ with $x^\prime,z^\prime$ we get the conclusion.
\end{proof}

We make the following assumption.
\begin{hypothesis} \label{hp:psi}
    For almost every $t\in [0,T]$ the map $\psi(t, \cdot ,  \cdot  )$ is G\^ateaux
     differentiable on $X_\eta \times \Xi^*$ and the maps $(x,h,z) \mapsto \nabla_x \psi(t,x,z)[h]$ and $(x,z,\zeta) \mapsto \nabla_z \psi(t,x,z)[\zeta]$ are continuous on $X_\eta \times X \times \Xi^*$ and $X_\eta \times \Xi^* \times \Xi^*$ respectively.
\end{hypothesis}
\begin{rem}\label{rem:gradpsi}
Under the above assumption we immediately deduce the following estimates:
\begin{align*}
     |\nabla_x \psi(t,x,z)[h]|\leq C(1+\|x\|_\eta)|h| \qquad \textrm{and} \qquad 
     |\nabla_z \psi(t,x,z)[\zeta]|\leq C|\zeta|, \quad h\in X,\zeta\in \Xi^*.
\end{align*}
\end{rem}
We notice that Hypothesis \ref{hp:psi} involves condition on the function $\psi$ and not on the functions $l$ and $r$ that determine $\psi$. However, Hypothesis \ref{hp:psi} can be verified in concrete situations (for an example, see, for instance, \cite[Ex. 2.7.1]{FuTe/2002}). 

The backward equation \eqref{eq:bsde1} is understood in the usual way: we look for a pair of processes $(Y,Z)$, progressively measurable, which, for any $t \in [s,T]$
\begin{align*}
      Y(t)+ \int_t^T Z(\tau) \dd W(\tau)= \phi(\bx(T;s,x)) -\int_t^T \psi(\tau,\bx(\tau;t,x),Z(\tau)\dd \tau.
\end{align*} 

We can now state the following result on existence, uniqueness and smoothness of the solution 
to equation \eqref{eq:bsde1}:
\begin{prop}\label{prop:bsde}
     \begin{enumerate}
       \item  For all $x\in X_\eta, \ s\in [0,T]$ and $p\in [0,\infty)$ there exists a unique pair of processes $(Y,Z)$ with $Y \in L^p(\Omega,C([s,T];\R))$, $Z\in L^p(\Omega, L^2(s,T;\Xi^*))$
   solving the equation \eqref{eq:bsde1}. In the following we will denote such a solution
by $Y(\cdot ; s,x)$ and $Z(\cdot ; s,x)$.
      \item The map $(s,x) \mapsto (Y( \cdot  ; s,x),Z(\cdot  ; s,x))$ is continuous from
       $[0,T] \times X_\eta$ to $L^p(\Omega,C([s,T];\R)) \times L^p(\Omega; L^2(s,T;\Xi^*))$.
      \item For all $s\in [0,T]$ the map $x \mapsto  (Y(\cdot ; s,x),Z(\cdot; s,x))$ is 
    G\^ateaux differentiable as a map from $X_\eta$ to $L^p(\Omega,C([s,T];\R)) \times
    L^p(\Omega; L^2(s,T;\Xi^*))$. Moreover, the map $(t,x,h) \mapsto 
    (\nabla_x Y(\ \cdot \ ; t,x),\nabla_x Z(\ \cdot \ ; t,x))$ is continuous from $[0,T] \times 
    X_\eta \times X$ to $ L^p(\Omega,C([s,T];\R)) \times
    L^p(\Omega; L^2(s,T;\Xi^*))$. 
    \item For any $x\in X_\eta, h\in X$, the pair of processes $(\nabla_x Y(\cdot ;s,x),\nabla_x Z( \cdot ; s,x))$ satisfies the equation
    \begin{align*}
          \begin{cases}
      \dd \nabla_x Y(\tau;s,x)[h]= \nabla_x \psi(\tau,\bx(\tau;s,x),Z(\tau)) \nabla_x X(\tau; s,x)[h]\dd \tau  \\
         \quad \quad + 
     \nabla_x \psi(\tau,\bx(\tau;s,x),Z(\tau)) \nabla_x Z(\tau; s,x)[h]\dd \tau + \nabla_x Z(\tau; s,x)[h]\dd W(\tau), \\
      \nabla_xY(T;s,x)[h] = \nabla_x\phi(\bx(T;s,x))\nabla_x X(T; s,x)[h],  
   \end{cases}
    \end{align*}
    for $\tau \in [s,T]$, $s\in [0,T]$.
     \end{enumerate}
\end{prop}

\begin{proof}
     The claim follows directly from Proposition 4.8 in Fuhrman and Tessitore \cite{FuTe/2002},
from the differentiability stated in Proposition \ref{prop:dif-state-eq} and the chain rule 
(as stated in Lemma 2.1 in \cite{FuTe/2002}). 
\end{proof}

As in the previuous section, starting from the G\^ateaux derivatives of $Y$ and $Z$, 
we introduce suitable auxiliary processes which will allow ourselves to express
$Z$  in terms of $\nabla Y$ and $(I-B)^{1-\theta} $ and then get the fundamental relation for the optimal control problem introduced in Section \ref{sec:intr}.
\begin{prop}\label{prop:PiQ}
      For every $p\geq 2$, $s\in [0,T], \ x\in X_\eta, h\in X$ there exists two processes 
     \begin{align*}
          \left\{ \Pi(t;s,x)[h]: \ t\in [0,T]\right\} \quad \textrm{and} \quad 
          \left\{ Q(t;s,x)[h]: \ t\in [0,T]\right\} 
     \end{align*}
     with $\Pi(\cdot;s,x)[h] \in L^p(\Omega;C([0,T];X_\eta))$ and 
    $Q(\cdot;s,x)[h] \in L^p(\Omega;C([0,T];\Xi^*))$ such that if
      $s\in [0,T)$, $x\in X_\eta$ and  $h\in \mathcal{D}$, where
    $$  \mathcal{D}:= \left\{h\in X_{1-\theta} \subset X: \ (I_B)^{1-\theta}h \in X_\eta\right\},
    $$ then $\PP$-a.s. the following identifications hold:
     \begin{align}
            \Pi(t;s,x)[h] =
           \begin{cases}
                      \nabla_x Y(t;s,x)(I-B)^{1-\theta}[h], & \textrm{for all } t\in [s,T];\\
                       \nabla_x Y(s;s,x)(I-B)^{1-\theta}[h], &\textrm{for all } t\in [0,s);\\
          \end{cases}\label{eq:rapPi}\\
           Q(t;s,x)[h]= \begin{cases}
                  \nabla_x Z(t;s,x)(I-B)^{1-\theta}[h], & \textrm{for a.e.} t\in [s,T]\\
                  0, & \textrm{for all } t\in [0,s).
            \end{cases}\label{eq:rapQ}
     \end{align}
    Moreover the map $(s,x,h) \mapsto \Pi(\cdot;s,x)[h]$ is continuous from
    $[0,T] \times X_\eta \times X$ to $L^p(\Omega; C([0,T];\R))$ and the map
     $(s,x,h) \mapsto Q(\cdot;s,x)[h]$ is continuous from 
     $[0,T]\times X_\eta \times X$ into $L^p(\Omega;C([0,T];\Xi^*))$ and both maps are linear 
      with respect to $h$. 
      Finally, there exists a positive constant  $C$ such that 
     \begin{multline}\label{eq:ineqPiQ}
       \bE \sup_{t\in [s,T]} \|\Pi(t;s,x)[h]\|_\eta^p + \bE \int_s^T |Q(t;s,x)[h]|^2_{\Xi^*} \dd t \\
          \leq C (T-s)^{(\theta-1-\eta)p}(1+\| x\|_\eta)^p |h|_X^p. 
     \end{multline}
\end{prop}

\begin{proof}
    As in the proof of Proposition \ref{prop:theta}, we introduce a suitable stochastic differential equation which should give the pair $(\Pi(\cdot;s,x)[h],Q(\cdot;s,x)[h])$; more precisely, for fixed $s\in [0,T], x\in X_\eta$ and $h\in X$ we consider the following backward stochastic differential equation on the unknown $(\Pi(\cdot;s,x)[h],Q(\cdot;s,x)[h])$:
\begin{align}\label{eq:bsdePiQ}
   \begin{cases}
       \dd \Pi(t;s,x)[h] = \nu(t;s,x) h\dd t + \nabla_z \psi(t,\bx(t;s,x),Z(t;s,x)) Q(t;s,x)[h] \dd s\\
        \qquad \qquad \qquad   + \ Q(t;s,x)[h] \dd W(s)  \qquad \qquad t\in [s,T]\\
      \Pi(T;s,x)[h]= \eta(s,x)h, 
      \end{cases}
\end{align}
where 
\begin{align*}
     \nu(t;s,x) h&= {\mathbf 1}_{[s,T]}(t) \nabla_x \psi(t,\bx(t;s,x),Z(t;s,x))\left( \Theta(t;s,x)[h]
     +
    (I-B)^{1-\theta}e^{(t-s)B} h\right),\\
     \eta(s,x) h &= \nabla_x \phi( \bx(T;s,x))\left( \Theta(T;t,x)[h]+ (I-B)^{1-\theta}e^{(T-s)B}h\right).
\end{align*}
    We notice that the processes $\nu(\cdot;s,x)h$ and $\eta(\cdot;s,x)h$ are well-defined 
for every $h\in X$. 
     Moreover, recalling the estimates
on $\bx(\cdot;s,x)$ obtained in Proposition \ref{thm:exuni}, the properties of  $\nabla_x \psi(t,x,z)$ (see Remark \ref{rem:gradpsi}) and the bound \eqref{eq:thetabound} on $\Theta(\cdot;s,x)h$ proved in \ref{prop:theta}, we get
  \begin{align*}
      &\bE \left(\int_0^T |\nu(t;s,x)h|^2 \dd t \right)^{p/2} \\
     \leq& \  C_1\  \bE \left( \int_s^T  (1+ \|\bx(t;s,x)\|_\eta)^2\left(\|\Theta(t;s,x) h\|_\eta + 
     (t-s)^{\theta-1-\eta}|h|_X\right)^2 \dd t \right)^{p/2}\\
    \leq & \ C_2 (1+\|x\|_\eta)^{p}\  \left[(T-s)^{p/2} + (T-s)^{(2(\theta-\eta)-1)p/2}\right]|h|^p_X\\
   \leq & \ C_3 (1+\|x\|_\eta)^{p}|h|^p_X.
   \end{align*}
  where $C_i,\ i=1,2,3$ are suitable constants independent on $t,x$ and $h$. 
In the same way, again by \eqref{eq:exuni}, \eqref{eq:thetabound} and by hypothesis \ref{hp:l,phi} 
we get
\begin{align*}
     \bE |\eta(s,x)h|^{p} & \leq C\bE \left[(1+ \|\bx(T;s,x)\|) (\|\Theta(t;s,x) h\|_\eta+ (T-s)^{\theta-1-\eta})|h|_X\right]^p \\
   & \leq C (1+\|x\|_\eta)^p (T-s)^{(\theta-1-\eta)p}|h|^p_X.
\end{align*}
  Hence, recalling Proposition 4.3 in Fuhrman and Tessitore \cite{FuTe/2002}, we obtain that 
for every $s\in [0,T],x\in X_\eta,h\in X$
there exists a unique pair of processes $(\Pi(\cdot;s,x)[h],Q(\cdot;s,x)[h])$ solving the  BSDE \eqref{eq:bsdePiQ}
such that $\Pi\in L^p(\Omega; C([0,T];X_\eta))$ and $Q\in L^p(\Omega;L^2(0,T;\Xi^*))$. Moreover, inequality \eqref{eq:ineqPiQ} holds.

Reasoning as in the proof of Proposition \ref{prop:theta}, when $h \in \mathcal{D}$ we get the representation given in
\eqref{eq:rapPi} and \eqref{eq:rapQ}. Clearly, the map $h\mapsto (\Pi(\cdot;s,x)[h],Q(\cdot;s,x)[h])$
is linear.

Now we prove its continuity. By estimate \eqref{eq:ineqPiQ}, it is sufficient to prove
that it is continuous with respect to $s$ and $x$ for any fixed $h\in X$, and, again by Proposition
4.3 in Fuhrman and Tessitore \cite{FuTe/2002}, it is enough to prove that for all $h\in X, s_n \to s \in [0,T)$,
$x_n \to x \in X_\eta$ 
we have
\begin{align*}
      I_1^{n} + I_2^{n}+ I_3^{n}  & := \bE \left( \int_0^T \left\| \nabla_z \psi (t,\bx^n(t),Z^n(t))-
      \nabla_z\psi(t,\bx(t),Z(t))\right\|_{L(\Xi)}^2 \dd t \right)^{p/2}\\
     &+ \bE \left( \int_0^T |\nu(t;s_n,x_n)h-\nu(t;s,x)h|_X^2 \dd t \right)^{p/2} \\
     &+ \bE \ |\eta(s_n,x_n)h-\eta(s,x)h|^p 
     \ \longrightarrow \ 0, \quad n\to \infty,
\end{align*}
where we set $\bx^n(t)=\bx(t;s_n,x_n)$ and $Z^n(t)=Z(t;s_n,x_n)$. 
We prove that the three integrals goes to $0$ separately. To this end we notice that, by the continuous dependence of $\bx$ and $Z$ on the initial data, we have
$\bx^n \to \bx$ and $Z^n \to Z$, respectively in $L^p(\Omega;C([0,T];X_\eta))$ and $L^p(\Omega;L^2(0,T;\Xi^*))$.
Thus from each subsequence in $\bN$ we can extract a subsequence 
 $\left\{ n_i: i\in \bN\right\}$ for which $\sum_{i} \|\bx^{n_i}-\bx\|_{L^p(\Omega;C([0,T];X_\eta)}<\infty$  and  the series $\sum_{i} \bx^{n_i}-\bx$ converges $\PP$-a.s. and for all $t\in [0,T]$ to an element $\hat{\bx}$  in 
$L^p(\Omega;C([0,T];X_\eta))$, for which  
$\|\bx^{n_i}(t)\|_\eta \leq \|\bx(t)\|_\eta + \|\hat{\bx}(t)\|_\eta$, 
$\PP$-a.s. for all $t\in [0,T]$ for $t\in [0,T]$. 
Similarly, there exists $\hat{Z}$ in $L^p(\Omega;L^2(0,T;\Xi^*))$ such that 
 $\|Z^{n_i}(t)\|_{\Xi^*} \leq \|Z(t)\|_{\Xi^*}  + \|\hat{Z}(t)\|_{\Xi^*} $,  $\PP$-a.s. for all $t\in [0,T]$. 
  
Now we use the above claim to prove that $I_i^{n} \to 0$, as $n\to \infty, i=1,2,3$.
Let us consider $I_1^{n}$. By the continuity assumptions stated in Hypothesis \ref{hp:psi}  and the convergence of $(\bx^{n_i},Z^{n_i})$ to $(\bx,Z)$, $\PP$-a.s. for all $t\in [0,T]$ we have that
  $\nabla_z \psi(t,\bx^{n_i}(t),Z^{n_i}(t))$ converges to $\nabla_z \psi(t,\bx(t),Z(t))$, $\PP$-a.s  for all $t\in [0,T]$. Further, by Remark \ref{rem:gradpsi},
the function in the integral $I_1$ is bounded by a constant independent on $n\in \bN$. Hence, by the dominated convergence theorem
we have
\begin{align*}
      \bE \left( \int_0^T \left\| \nabla_z \psi (t,\bx^{n_i}(t),Z^{n_i}(t))-
      \nabla_z\psi(t,\bx(t),Z(t))\right\|_{L(\Xi)}^2 \dd t \right)^{p/2} \ \longrightarrow \  0.
\end{align*} 
To prove that $I_2^n \to 0$ as $n\to \infty$ we define
\begin{align*}
      V^{n}(t):&= {\bf 1}_{[s_n,T]}(t) \left( \Theta(t;s_n,x_n)[h]+ (I-B)^{1-\theta}e^{(t-s_n)B}h\right)
\intertext{and}
    V(t):&= {\bf 1}_{[s,T]}(t) \left( \Theta(t;s,x)[h]+ (I-B)^{1-\theta}e^{(t-s)B}h\right),
\end{align*}
and we notice that $\nu(t;s_n,x_n)=\nabla_x \psi(t,\bx^n(t),Z^n(t)) V^n(t)$ and
  $\nu(t;s,x)=\nabla_x \psi(t,\bx(t),Z(t)) V(t)$. Then
\begin{align*}
     I_2^n \leq & \ C_1 \bE \left(\int_0^T \left|\left(\nabla_x \psi(t; \bx^n(t),Z^n(t))-\nabla_x \psi(t;\bx(t),Z(t))\right) V(t)\right|^2\dd t\right)^{p/2} \\
      +& \  C_2 \bE \left( \int_0^T \left| \nabla_x \psi(\bx^n(t),Z^n(t))(V^n(t)-V(t))\right|^2\dd t\right)^{p/2}\\
   = & I_{21}^n+ I_{22}^n.
\end{align*}
Taking into account the continuity assumption  and the bound on $\nabla_x \psi(t,x,z)$ (see respectively Hypothesis \ref{hp:psi} and Remark \ref{rem:gradpsi}), we have 
\begin{align*}
   &  \bE\sup_{t\in [0,T]} \left|\left(\nabla_x \psi(t; \bx^n(t),Z^n(t))-\nabla_x \psi(t;\bx(t),Z(t))\right) \right|^2\\
&   \quad \quad \leq \   C\ \bE \sup_{t\in [0,T]} ( 1+ \|\bx^n (t)\|_\eta + \| \bx(t)\|_\eta)^2 \\
&  \quad \quad   \leq \  C(1+\|x_n\|_\eta+ \|x\|_\eta)^2 
\end{align*}
Further $|V(t)|^2 \leq c {\bf 1}_{[s,T]}(t) (1+(t-s)^{2(\theta-\eta-1)})$ and $2(\theta-\eta-1)>-1$
(since $\theta-\eta>1/2$. Hence, reasoning as it was done for $I_1$ we can apply the dominated convergence theorem and obtain that $I_{21}^n \to 0$, as $n\to \infty$. To show that 
$I_{22}^n \to 0$, as $n\to \infty$, we apply H\"older inequality
 to $\bE \int_0^T  \left| \nabla_x \psi(\bx^n(t),Z^n(t))(V^n(t)-V(t))\right|^2\dd t$ to get
\begin{align*}
  & \bE\left( \int_0^T \left| \nabla_x \psi(\bx^n(t),Z^n(t))(V^n(t)-V(t))\right|^2\dd t \right)^{p/2} \\ 
  \leq & 
  \ \bE \left( \left( \int_0^T |\nabla_x \psi(t,\bx^n(t),Z^n(t))|^{2r} \dd t\right)^{1/r}
   \left( \int_0^T |V^n(t)-V(t)|^{2q} \dd t\right)^{1/q}\right)^{p/2} \\
   \leq & \ C \ \bE \left(\left(\int_0^T \left( 1+ \|\bx^n(t)\|_\eta\right)^{2r}\dd t \right)^{p/2r} 
    \left(\int_0^T |V^n(t)-V(t)|^{2q}\dd t \right)^{p/2q}\right) \\
    \leq & \ C \ \bE \left( \left(1+\sup_{t\in[0,T]}\|\bx^n(t)\|_{\eta}^p\right)
        \left(\int_0^T |V^n(t)-V(t)|^{2q}\dd t \right)^{p/2q}\right)  \\
    \leq & \ C\left(1+\bE \sup_{t\in [0,T]}\|\bx^n(t)\|_\eta\right)^p T^{p/r} \,\bE \left(\int_0^T |V^n(t)-V(t)|^{2q}\dd t \right)^{p/q} 
\end{align*}
for every $r,q$ such that $1/r+1/q=1$.  By the above inequality we see that 
$I_{22}^n \to 0, \ n\to \infty$ if $V^n \to V$ in $L^p(\Omega;L^{2q}(0,T;X_\eta))$. 
To prove this limit we first note that 
\begin{align*}
    \bE \left(\int_0^T \left|{\bf 1}_{[s_n,T]}(t) \Theta(t;s_n,x_n)[h] - {\bf 1}_{[s,T]}(t) \Theta(t;s,x)[h]\right|^{2q} \right)^{p/q} \ \longrightarrow \ 0,
\end{align*}
since the map $(s,x) \mapsto \Theta(\cdot; s,x)$ is continuous with values in $L^p(\Omega;C([0,T];X_\eta))$. Hence it remains to show that 
\begin{align*}
     {\bf 1}_{[s_n,T]}(\cdot) (I-B)^{1-\theta}e^{(\cdot - s_n)B}h  \ \longrightarrow \ {\bf 1}_{[s,T]}
    (\cdot) (I-B)^{1-\theta} e^{(\cdot-s)B}h
\end{align*}
in $L^{2q}([0,T];X_\eta)$. To this end we note that for all $s_n^+ \searrow s$ and $s_n^- \nearrow s$
we have 
\begin{multline*}
     \int_0^T \left| (I-B)^{1-\theta}\left(e^{(t - s_n)B}- e^{(t-s)B} \right)h\right|^{2q} \dd t
      \\
     = \int_{s_n^-}^{s_n^+} \left|(I-B)^{1-\theta}e^{(t-s_n^-)B}h \right|^{2q}\dd t +
      \int_{s_n^+}^T \left| (I-B)^{1-\theta}e^{(t-s_n^+)B}\left( e^{(s_n^+-s_n^-)B}-I\right)h\right|^{2q} \dd t 
     \\
    \leq |h|^{2q}\int_{s_n^-}^{s_n^+} (t-s_n^-)^{2q(\theta-1-\eta)} \dd t + 
    \left| \left( e^{(s_n^+-s_n^-)B}-I\right)h\right|^{2q} \int_{s_n^+}^T (t-s_n^+)^{2q(\theta-1-\eta)} \dd t \ \longrightarrow \ 0, 
\end{multline*}
provided $q$ is sufficiently close to $1$, since $\theta-\eta >1/2$. 

In a similar way as it has been done for $I_2^n$, we can prove that 
$I_3^n \to 0, \ n\to \infty$.   
\end{proof}

We are now in the position to give a meaning to the expression $\nabla_x Y(t;s,x)(I-B)^{1-\theta}$ and, successively, to identify it with the process $Z(t;s,x)$.
\begin{cor}\label{cor:quad-var}
     Setting $v(s,x)=Y(s;s,x)$, we have $v\in C([0,T]\times X_\eta;\R)$ and there exists a constant 
   $C$ such that $|v(s,x)|\leq C(1+|x|)^2$, $t\in [0,T],\ x\in X_\eta$. Moreover $v$ is G\^ateaux differentiable with respect to $x$ on $[0,T] \times X_\eta$ and the map $(s;x,h) \mapsto \nabla v(s,x)[h]$ is continuous. 

Moreover, for $s\in [0,T)$ and $x\in X_\eta$ the linear operator $h\mapsto \nabla v(s,x)(I-B)^{1-\theta}h$ - a priori defined for $h\in \mathcal{D}$ (with $\mathcal{D}$ as in 
\ref{eq:D}) - has an extension to a bounded linear operator from $X$ into $\R$, that we
denote by $[\nabla v(I-B)^{1-\theta}](s,x)$.   

Finally, the map $(s,x,h) \mapsto [\nabla v(I-B)^{1-\theta}](s,x)$ is continuous as a mapping from
 $[0,T) \times X_\eta \times X$ into $\R$ and there exists $C>0$ such that
\begin{multline}\label{eq:stima-grad-Y}
      |[\nabla v(I-B)^{1-\theta}(s,x)]h|_X\leq C (T-s)^{(\theta-\eta-1)}(1+\|x\|_\eta)|h|_X, \\
    \quad t\in [0,T), \ x\in X_\eta, \ h\in X.
\end{multline}
\end{cor}

\begin{proof}
    We recall that $Y(s;s,x)$ is deterministic.  
  
Since the map $(s,x) \mapsto Y(\cdot;s,x)$ is continuous with values in 
 $L^p(\Omega;C([0,T];\R))$, $p\geq 2$, then the map $(s,x) \mapsto Y(s;s,x)$ is continuous with values in $L^p(\Omega;\R)$ and so the map $(s,x) \mapsto \bE Y(s;s,x)=Y(s;s,x)=v(s,x)$ is continuous with values in $\R$. 

Similarly $\nabla_x v(s,x)= \bE \nabla_x Y(s;s,x)$ exists and has the required properties,
by Proposition \ref{prop:bsde}.  Next we notice that $\Pi(s;s,x)h=\nabla_x Y(s;s,x)(I-B)^{1-\theta}h$. The existence of the required extensions and its continuity are direct consequence of Proposition \ref{prop:PiQ} and estimate \ref{eq:stima-grad-Y} follows directly from \eqref{eq:ineqPiQ}. 
\end{proof}

\begin{cor}
    For every $t\in [0,T]$, $x\in X_\eta$ we have, $\PP$-a.s.
   \begin{gather}
          Y(t;s,x)=v(t;\bx(t;s,x)), \qquad \textrm{for all } t\in [s,T], \label{eq:i1}\\
         Z(t;s,x)=[\nabla v(I-B)^{1-\theta}](t;\bx(t;s,x))(I-B)^{\theta}P g, \qquad \textrm{for almost all } t\in [s,T] \label{eq:i2}.
    \end{gather}
\end{cor}

\begin{proof}
     We start from the well-known equality: for $0\leq s \leq r \leq T$, $\PP$-a.s. 
   $$
      \bx(t;s,x)=\bx(t;r,\bx(r;s,x)), \qquad \textrm{for all } t\in [s,T].
   $$
  It follows easily from the uniqueness of the backward equation \eqref{eq:bsde1} that 
   $\PP$-a.s. 
   $$
       Y(t;s,x)=Y(t;r,\bx(r;s,x)), \qquad  \textrm{for all } t\in [s,T].
$$
Setting $s=r$ we arrive at \eqref{eq:i1}. 

To prove \eqref{eq:i2} we consider the joint quadratic variation of $(Y(t;s,x))_{t\in [s,T] }$
and $W$ on an arbitrary interval $[s,t] \subset [s,T)$; from the backward equation
\eqref{eq:bsde1} we deduce that this is equal to $\int_s^t Z(r;s,x)\dd r$. On the other side, 
the same result can be obtained by considering the joint quadratic variation 
of $(v(t,\bx(t;s,x)))_{t\in [s,T]}$ and $W$. Now by an application of Proposition \ref{prop:quad-var}
(whose assumptions hold true by Corollary \ref{cor:quad-var}) leads to the identity
$$
     \int_s^t Z(r;s,x)\dd r = \int_s^t [\nabla v(I-B)^{1-\theta}](r,\bx(t;s,x))(I-B)^\theta P g \dd r,
$$
and \eqref{eq:i2} is proved.
\end{proof}

\section{The Hamilton-Jacobi-Bellman equation}\label{sec:HJB}
Let us consider again the solution $\bx(t;s,x)$ of equation \eqref{eq:state-eq-2} and denote by 
$P_{s,t}$ its transition semigroup:
$$
      P_{s,t}[h](x)=\bE h(\bx(t;s,x)), \qquad x\in X_\eta, 0\leq s \leq t \leq T,
$$
for any bounded measurable $h: X \to \R$. We notice that by the bound \eqref{eq:exuni} this formula is meaningful for every $h$ with polynomial growth. In the following $P_{s,t}$ will be considered as an operator acting on this class of functions.

Let us denote by $\cL_t$ the generator of $P_{s,t}$:
\begin{multline*}
     \cL_t [h](x)=\frac{1}{2}{\rm Tr} [(I-B)Pg\nabla^2h(x)g^*p^*(I-B)^*]\\
    + \langle Bx+(I-B)P f(t,Jx),\nabla h(x)\rangle,
\end{multline*}
where $\nabla h$ and $\nabla^2 h$ are first and second G\^ateaux derivatives 
of $h$ at the point $x\in X$ (here we are identified with elements of $X$ and $L(X)$ respectively).
This definition is formal, since it involves the terms $(I-B)Pg$ and $(I-B)P f$ which - {\emph a priori} - are not defined as elements of $L(X)$ and the domain of $\cL_t$ is not specified.

The Hamilton-Jacobi-Bellman \eqref{eq:HJB} equation for the optimal control is 
\begin{align}\label{eq:HJB}\tag{$HJB$}
    \begin{cases}
        \frac{\partial }{\partial t}v(t,x)+\cL_t[v(t,\cdot)](x)= \psi(t,x,\nabla v(t,x)(I-B)Pg), \quad t\in [0,T], \ x\in X,\\
       v(T,x)=\phi(x).
    \end{cases}
\end{align}
This is a nonlinear parabolic equation for the unknown function $v: [0,T] \times X_\eta \to \R$. 
We define the notion of solution of the \eqref{eq:HJB} by means of the variation of constant formula:
\begin{defn}\label{def:HJB}
    We say that a function $v: [0,T]\times X_\eta \to \R$ is a mild solution of the Hamilton - Jacobi - Bellman equation \eqref{eq:HJB} if the following conditions hold:
\begin{enumerate}
    \item $v\in C([0,T]\times X;\R)$ and there exist constants $C,m\geq 0$ such that 
    $|v(t,x)|\leq C(1+|x|)^m, \ t\in[0,T], \ x\in X$;
    \item $v$ is G\^ateaux differentiable with respect to $x$ on $[0,T)\times X_\eta$
     and the map $(t,x,h)\mapsto \nabla v(t,x)[h]$ is continuous $[0,T)\times X_\eta \times X \to \R$;
    \item \label{it:boundv} For all $t\in [0,T)$ and $x\in X_\eta$ the linear operator $k \mapsto \nabla v(t,x)(I-B)^{1-\theta} k$ (a priori defined for $k\in \mathcal{D})$ has an extension to a bounded linear operator on $X \to \R$, that we denote by $[\nabla v(I-B)^{1-\theta}](t,x)$.
   
Moreover the map $(t,x,k)\mapsto [\nabla v (I-B)^{1-\theta}](t,x)k$ is continuous $[0,T) \times X_\eta \times X \to \R$ and there exist constants $C,m\geq 0, \ \kappa\in [0,1)$ such that
\begin{align*}
   \| [\nabla v (I-B)^{1-\theta}](t,x)\|_{L(X)} \leq C(T-t)^{-\kappa}(1+\|x\|_\eta)^m, \qquad t\in [0,T), x\in X_\eta.
\end{align*}
  \item The following equality holds for every $t\in [0,T]$, $x\in X_\eta$:
    \begin{align}\label{eq:solvar}
         v(t,x)=P_{t,T}[\phi](x)-\int_t^T P_{t,r} \left[ \psi(r, \cdot, [\nabla v (I-B)^{1-\theta}](r,\cdot)(I-B)^\theta P g\right](x) \dd r.
    \end{align}
\end{enumerate} 
\end{defn}

\begin{rem}
    We notice that Proposition \ref{prop:psi} implies that $|\psi(t,x,z)|\leq C(1+|z|+|x|^2)$, so that 
 if $v$ is a function satisfying the bound required in \ref{it:boundv} of the above Definition
  we have
\begin{align*}
     \left| \psi(t,x,[\nabla v (I-B)^{1-\theta}](t,x)(I-B)^{\theta} P g )\right|\leq C(T-t)^{-\kappa}
(1+\|x\|_\eta)^{m+2}
\end{align*}
and formula \eqref{eq:solvar} is meaningful.
\end{rem}
Now we are ready to prove that the solution of the equation \eqref{eq:HJB} can be defined by means of the solution of the BSDE associated with the control problem \eqref{eq:state-eq-2}.
\begin{thm}\label{thm:ident}
    Assume Hypothesis \ref{hp:a,A,f,g,r,W}, \ref{hp:dato-in} and \ref{hp:l,phi}, there exists a unique mild solution of the Hamilton - Jacobi - Bellman equation \ref{eq:HJB}. The solution is given by the formula
\begin{align}\label{eq:defv}
     v(t,x)=Y(t;t,x),
\end{align} 
where $(\bx,Y,Z)$ is the solution of the forward-backward system \eqref{eq:state-eq-2} and \eqref{eq:bsde1}. 
\end{thm}
\begin{proof}
     We start by proving existence. By Corollary \ref{cor:quad-var} the function $v$ defined as in 
    \eqref{eq:defv} has the regularity properties stated in Definition \ref{def:HJB}. In order to verify that equality \eqref{eq:solvar} holds we first fix $t\in [0,T]$ and $x\in X_\eta$. We notice that 
\begin{multline*}
     \psi(t,\cdot, v(t,\cdot), [\nabla v (I-B)^{1-\theta}](t,\cdot)(I-B)^{\theta}P g)(x) \\
=
         \psi(t,\cdot, Y(t;t,\cdot), [\nabla Y (I-B)^{1-\theta}](t,\cdot)(I-B)^{\theta}P g)(x)
\end{multline*}
and we recall that 
\begin{align*}
     [\nabla v(I-B)^{1-\theta}](t;\bx(t;s,x))(I-B)^{\theta}Pg=Z(t;s,x).
\end{align*}
Hence 
\begin{multline}\label{eq:psiPtT}
     P_{s,t}\left[ \psi(t,\cdot, v(t,\cdot), [\nabla v (I-B)^{1-\theta}](t,\cdot)(I-B)^{\theta}P g)\right](x) \\
  = \bE \left[ \psi(t,\bx(t;t,x), Y(t;t,\bx(t;t,x)), Z(t;t,(\bx(t;t,x)))\right].
\end{multline}
On the other hand, the backward equation gives
\begin{align*}
    Y(t;t,x)+\int_t^T Z(r;t,x)\dd r = \phi(\bx(T;t,x))-\int_t^T \psi(r,\bx(r;t,x),Z(r;t,x)) \dd r.
\end{align*}
Taking the expectation we obtain 
\begin{align*}
    v(t,x)=P_{t,T}[\phi](x)-\bE \int_t^T \psi(r,\bx(r;t,x),Z(r;t,x)) \dd r 
\end{align*}
and substituting in the integral the expression obtained in \eqref{eq:psiPtT} we get the required equality \eqref{eq:solvar}.

Now we consider uniqueness of the solution. Let $v$ denote a mild solution. We look for a convenient expression for the process $v(t,\bx(t;s,x))$, $t\in [s,T]$. By \eqref{eq:solvar}
and the definition of $P_{s,t}$ we have
\begin{multline*}
    v(t,x)
    = \bE \left[\phi(\bx(T;s,x))\right]\\
  -\int_t^T 
      \bE \left[ \psi(r, \bx(r;s,x),[\nabla v (I-B)^{1-\theta}] (r,\bx(r;s,x))(I-B)^{\theta}P g )\right]
    \dd r. 
\end{multline*}
 Since $\bx(t;s,x)$ is $\cF_t$-measurable, we can replace the expectation by the conditional expectation given $\cF_t$:
\begin{multline*}
    v(t,x)
    = \bE ^{\cF_t}\left[\phi(\bx(T;s,x))\right]\\
  -\bE^{\cF_t}\left[ \int_t^T 
       \psi(r, \bx(r;s,x),[\nabla v (I-B)^{1-\theta}] (r,\bx(r;s,x))(I-B)^{\theta}P g 
    \dd r\right]. 
\end{multline*}
Moreover, recalling that for any $r\in [t,T]$ the equality 
$$
   \bx(r;t,\bx(t;s,x))=\bx(r;s,x)
$$ hold $\PP$-a.s.  we can replace $x$ by $\bx(t;s,x)$ to get:
\begin{align*}
     v(t,\bx(t;s,x)) &= \bE ^{\cF_t}\left[\phi(\bx(T;s,x))\right] \\
     &  -  \bE ^{\cF_t}\left[\int_t^T 
     \psi(r, \bx(r;t,x),[\nabla v (I-B)^{1-\theta}] (r,\bx(r;t,x))(I-B)^{\theta}P g 
    \dd r\right]\\
     &= \bE^{\cF_t}[\xi] \\
     &+\int_s^t  \psi(r, \bx(r;t,x),[\nabla v (I-B)^{1-\theta}] (r,\bx(r;t,x))(I-B)^{\theta}P g 
    \dd r,
\end{align*}
where we have defined 
\begin{multline*}
     \xi:= \phi(\bx(T;t,x))\\
     - \int_s^T  \psi(r, \bx(r;t,x),[\nabla v (I-B)^{1-\theta}] (r,\bx(r;t,x))(I-B)^{\theta}P g )\dd r.
\end{multline*}
 Now we notice that $\bE^{\cF_s}=\bE^{\cF_s}[v(t,\bx(t;s,x))]= v(s,x)$. Since $\xi\in L^2(\Omega;\R)$ is $\cF_T$-measurable, by a well-known g theorem there exists $\tilde{Z} \in L^2_\cF(\Omega\times [s,T];L_2(\Xi;\R))$ such that 
   $\bE^{\cF_t}[\xi]=\int_s^t \tilde{Z}(r)\dd W(r)+ v(s,x)$ We conclude that the process $v(t,\bx(t;s,x))$, $t\in [s,T]$ is a real continuous semimartingale with canonical decomposition 
\begin{multline}\label{eq:decomp}
    v(t,\bx(t;s,x))=\int_s^t \tilde{Z}(r)\dd W(r) + v(s,x) \\
      + \int_s^t  \psi(r, \bx(r;t,x),[\nabla v (I-B)^{1-\theta}] (r,\bx(r;t,x))(I-B)^{\theta}P g 
    \dd r,
\end{multline}
into its continuous martingale part and continuous finite variation part. 

Now we compute the joint quadratic variation process of both sides of the above equality with $W$ on an arbitrary interval $[0,t] \subset [0,T)$. By the assumtpion made in
\ref{it:boundv} we have that there exists a constant $K_t$ such that 
$\|\nabla v(s,x)\|_{L(X)} \leq K_t (1+\|x\|_\eta)^m$, for $s\in [0,t]$, $x\in X_\eta$; then we can apply Proposition \ref{prop:quad-var} to conclude that the joint quadratic variation equals
\begin{align*}
    \int_0^t [\nabla v(I-B)^{1-\theta}](r,\bx(r;s,x))(I-B)^{\theta}P g\dd r.
\end{align*} 
Computing the joint quadratic variation of the left-hand side of \eqref{eq:decomp} with $W$ yields the identity
\begin{align*}
    \int_0^t [\nabla v(I-B)^{1-\theta}](r,\bx(r;t,x))(I-B)^\theta P g \dd r= \int_s^t \tilde{Z}(r)\dd r.
\end{align*}
Therefore, for a.a. $t\in [s,T]$, we have $\PP$-a.s. $[\nabla v(I-B)^{1-\theta}](r,\bx(r;s,x))(I-B)^\theta Pg =\tilde{Z}(r)$, so substituting into \eqref{eq:decomp} and taking into account that 
$v(T,\bx(T;s,x))=\phi(\bx(T;s,x))$ we obtain, for $t\in [s,T]$,
\begin{multline*}
     v(t,\bx(t;s,x))+ \int_t^T [\nabla v(I-B)^{1-\theta}](r,\bx(r;s,x))(I-B)^\theta Pg \dd W(r) \\
   = \phi(\bx(T;s,x)) - \int_t^T \psi(r, \bx(r;t,x),[\nabla v (I-B)^{1-\theta}] (r,\bx(r;t,x))(I-B)^{\theta}P g 
    \dd r.
\end{multline*}
Comparing with the backward equation \ref{eq:bsde1} we notice that the pairs 
\begin{align*}
    & (Y(t;s,x),Z(t;s,x))
\intertext{and}
  & (v(t,\bx(t;s,x)), [\nabla v(I-B)^{1-\theta}](r,\bx(r;s,x))(I-B)^\theta Pg =\tilde{Z}(r))
\end{align*}
solve the same equation. By uniqueness, we have $Y(t;s,x)=v(t;\bx(t;s,x)), t\in [s,T]$, and setting $t=s$ we obtain $Y(t;t,x)=v(t,x)$. 
\end{proof}

\section{Sinthesis of the optimal control}\label{sec:cont}
In this section we prooced with the study of the optimal control problem associated to the stochastic Volterra equation \eqref{eq:Volterra}. 
For fixed $x\in X_\eta$ an admissible control system $(a.c.s.)$ is given by 
$\bU=(\hat{\Omega};\hat{\cF},(\hat{\cF}_t)_{t\geq 0},\hat{\PP},(\hat{W}(t))_{t\geq 0},\hat{\gamma})$ where
\begin{enumerate}
   \item $(\hat{\Omega},\hat{\cF},\hat{\PP})$ is a complete probability space and $(\hat{\cF}_t)_{t\geq 0}$ 
is a filtration on it satisfying the usual conditions;
\item $\hat{W}(t))_{t\geq 0}$ is a cylindrical $\hat{\PP}$-Wiener process with values in $\Xi$ and adapted to the filtration $(\hat{\cF}_t)_{t\geq 0}$;
\item 
 $\hat{\gamma}(t) \in \cU$, $\hat{\PP}$-a.s. for a.a. $t\in [0,T]$ where $\cU$ is a fixed subset of $U$.
\end{enumerate} 
To each $(a.c.s.)$  $\bU$ we associate the mild solution $u^\bU$ of the Volterra equation
\begin{align*}
    \begin{cases}
       \frac{\dd}{\dd t} \int_{-\infty}^t a(t-s)u^\gamma(s)\dd s= A u^\bU(t)+ f(t,u(t))\\
         \qquad \qquad \qquad \qquad + g \,[\, r(t,u^\bU(t),\hat{\gamma}(t))
            + \dot{\hat{W}}(t)\, ], \qquad t\in [0,T]\\
       u^\gamma(t)=u_0(t), \qquad t\leq 0.
    \end{cases}
\end{align*}
We have showed in Section \ref{sec:anal-set} that we can associate to such equation
the controlled state equation
\begin{align*}
   \begin{cases}
    \dd \bx^\bU(t)= B \bx^\bU(t)\dd t+ (I-B)P \, f(t, J \bx^\bU(t)) \\
       \qquad \qquad \qquad + (I-B)P g (r(t,J \bx^\bU(t),\hat{\gamma}(t))\dd t+\dd \hat{W}(t)\\
     \bx^\bU(0)=x.
   \end{cases}
\end{align*}
Our purpose is to 
minimize a cost functional
of the form 
\begin{align*}
    \bJ(u_0,\bU)=\bE \int_0^T l(r,u^\bU(r),\hat{\gamma}(t))\dd t + \bE \phi(u^\bU(T)).
\end{align*}
To this end we will consider its translation 
in the state space setting, where the cost functional is of the form 
\begin{align*}
    \bJ(x,\bU)=\bE \int_0^T l(t,J\bx^\bU(t),\hat{\gamma}(t))\dd t + \bE  \phi(J \bx^\bU(T)).
\end{align*}
We will work under the assumptions given in Hypothesis \ref{hp:a,A,f,g,r,W}, \ref{hp:dato-in} and \ref{hp:l,phi}. 
We recall that the Hamiltonian corresponding to our control problem is given by
\begin{align*}
    \psi(t,x,z):=\inf_{\gamma \in \cU} \left\{ l(t,Jx,\gamma)+z r(t,Jx,\gamma): \ \gamma\in U\right\}
\end{align*}
and we introduce the set of minimizers of \eqref{eq:Ham}:
\begin{align*}
    \Gamma(t,x,z):=\left\{ \gamma \in U: \ l(t,Jx, \gamma)+ zr(t,Jx,\gamma)=\psi(t,x,z)\right\}.
\end{align*}
The existence of a minimizer for the Hamiltonian is not a direct consequence of our setting. Then we require it explicitly.
\begin{hypothesis}\label{hp:minim}
     $\Gamma(t,x,z)$ is not empty for for all $t\in [0,T]$, $x\in X_\eta$ and $z\in \Xi^\star$.
\end{hypothesis}

Under this assumption, by the Filippov Theorem (see, e.g. \cite[Theorem 8.2.10, p. 316]{AmbroProdi}) there exists a measurable selection of $\Gamma$, a Borel measurable function $\mu: [0,T]\times X_\eta \times \Xi \to U$ such that
\begin{multline}
     \psi(t,x,z)=l(t,Jx,\mu(t,x,z))+z r(t,Jx,\mu(t,x,z))\\
      t\in [0,T], \ x\in X_\eta, z\in \Xi^\star.
\end{multline}
Further, by $v$ we will denote the solution of the Hamilton-Jacobi-Bellman equation relative 
to the above stated problem 
\begin{align}\label{eq:HJB2}\tag{$HJB$}
    \begin{cases}
        \frac{\partial }{\partial t}v(t,x)+\cL_t[v(t,\cdot)](x)= \psi(t,x,\nabla v(t,x)(I-B)Pg), \quad t\in [0,T], \ x\in X,\\
       v(T,x)=\phi(x).
    \end{cases}
\end{align}
As it was shown in the previous section, \eqref{eq:HJB2} admits a unique mild solution.
  
As stated in the introduction, we wish to perform the standard synthesis of the optimal control problem, that is to say, which consist in proving that the solution of the \eqref{eq:HJB} equation is the value function of the control problem and allows to construct the optimal feedback law. 
To be able to use nonsmooth feedbacks we settle the problem in the framework of weak control problem. Again we follow the approach of Fuhrman and Tessitore \cite{FuTe/2002} with slight modifications. 

The first step is to prove the so called \emph{fundamental relation}, which gives a characterization of the value function of the control problem in terms of the solution of the Hamilton-Jacobi-Bellman equation. 
\begin{prop} \label{prop:fund-rel}
     Let $v$ be the solution of \eqref{eq:HJB}. For every admissible control system $\bU=(\hat{\Omega},\hat{\cF},(\hat{\cF}_t)_{t\geq 0},\hat{\PP}, \hat{\gamma})$ and for the corresponding trajectory $\bx^\bU$ starting at $x\in X_\eta$ we have
\begin{align*}
    \bJ(x,\hat{\gamma})&= v(0,x)
  \\&+\bE \int_0^T -\psi\left(\sigma,\bx^\bU(\sigma)),[\nabla v (I-B)^{1-\theta} ](\sigma,\bx^\bU(\sigma))(I-B)^{\theta}P g \right) \dd \sigma\\
   & +  \bE \int_0^T  [\nabla v(I-B)^{1-\theta}](\sigma,\bx^\bU(\sigma))(I-B)^\theta Pg r(\sigma,J\bx^\bU(r),\hat{\gamma}(\sigma))\dd \sigma \\
  & + \bE \int_0^T l(\sigma,J \bx^\bU(\sigma),\hat{\gamma}(\sigma)) \dd \sigma.
\end{align*} 
\end{prop}
\begin{proof}
     The proof follows from the same arguments used in the proof of Theorem
7.2 in \cite{FuTe/2002} and is, therefore, omitted. Just notice that in this case by Theorem
  \ref{thm:ident} we have $Z(t;s,x)=[\nabla v(I-B)^{1-\theta}](t,\bx(t;s,x))(I-B)^\theta P g$ and the role of $G$ in \cite[Theorem 7.2]{FuTe/2002} is here played by $(I-B)Pg$.
\end{proof}

A straightforward consequence of Proposition \ref{prop:fund-rel} is the so-called Verification Theorem.
\begin{cor}\label{cor:fund-rel}
   For every admissible control system $\bU=(\hat{\Omega},\hat{\cF},(\hat{\cF}_t)_{t\geq 0},\hat{\PP}, \hat{\gamma})$ and initial datum $x\in X_\eta$ we have $\bJ(x,\bU)\geq v(0,x)$, and the equality holds if and only if the following feedback law holds $\PP$-a.s. for almost every $t\in [0,T]$:
\begin{align*}
    \hat{\gamma}(t)&=\mu(t,\bx^{\bU}(t),[\nabla v (I-B)^{1-\theta}](t,\bx^{\bU}(t))(I-B)^\theta P g)\\
    & \qquad \qquad \qquad \qquad \hat{\PP}-a.s. \ for \ a.a. t\in [0,T],
\end{align*}
where $\bx^{\bU}$ is the trajectory starting at $x$ and corresponding to the control $\hat{\gamma}$. 
In this case the pair $(\hat{\gamma}(\cdot),\bx^\bU(\cdot))$ is optimal.
\end{cor}

We are now in the position to state the existence and uniqueness of the so-called closed loop equation, which is given by 
\begin{equation}\label{eq:feedback}
    \begin{cases} 
        \dd \bx^\gamma(t) = B \bx^\gamma(t)\dd t + (I-B)Pf(t,J\bx^\gamma(t))\\
      \qquad \qquad + (I-B)P g(r(t,J \bx^\gamma(t),\hat{\gamma}(t,\bx^\gamma(t))) \dd t + \dd W(t)) \\
      \bx(0)=x.
    \end{cases}
\end{equation}
where $\hat{\gamma}$ is given by
\begin{equation}
    \hat{\gamma}(t,x):=\mu(t,x,[\nabla v (I-B)^{1-\theta}](t,x)(I-B)^\theta Pg).
\end{equation}


The main result of this section thus reads as follows:
\begin{prop}
     For every $t\in [0,T], x\in X_\eta$, the closed loop equation \eqref{eq:feedback} admits a weak solution $(\hat{\Omega},\hat{\cF},(\hat{\cF}_t),\hat{\PP}, \hat{W},\bar{\bx}(t))_{t\geq 0})$ which is unique in law and setting
\begin{align*}
    \hat{\gamma}(t)=\mu(t,\bar{\bx}(t),[\nabla v (I-B)^{1-\theta}](t,\bar{\bx}(t))(I-B)^\theta P g)
\end{align*}
we obtain an optimal admissible control system $(\hat{\Omega},\hat{\cF},(\hat{\cF}_t),\hat{\PP}, \hat{W},\bar{\bx}(t))_{t\geq 0},\hat{\gamma})$.
\end{prop} 
      
\begin{proof}
    Let us take an arbitrary set up $(\Omega,\cF,(\cF_t),\PP,W)$ and consider the solution $(\bar{\bx}(t))_{t\geq 0}$ of the uncontrolled equation 
\begin{align}\label{eq:un-con-weak}
  \begin{cases} 
        \dd \bar{\bx}(t) = B \bar{\bx}(t)\dd t + (I-B)Pf(t,J\bar{\bx}(t))
          + (I-B)P g\dd W(t) \\
      \bar{\bx}(0)=x,
    \end{cases}
\end{align}
which exists in virtue of Theorem \eqref{thm:exuni}. 
Now we define the control
$$
    \hat{\gamma}(t)=\mu(t,\bar{\bx}(t),[\nabla v (I-B)^{1-\theta}](t,\bar{\bx}(t))(I-B)^\theta P g)
$$ 
and the process
$$
      \hat{W}(t):= W(t)-\int_s^{t\vee s} r(\sigma,\bx(\sigma),\hat{\gamma}(r)) \dd r, \quad  t\in [0,T].
$$
Since the function $r$ is bounded, by Girsanov theorem there exists a probability $\hat{\PP}$ on $(\Omega,\cF)$ equivalent to $\PP$, such that $\hat{W}$ is a $\hat{\PP}$-Wiener process with respect to $(\cF_t)$. Rewriting equation \eqref{eq:un-con-weak} in terms of $\hat{W}$ we conclude that $\bar{\bx}$ is the required solution of \eqref{eq:feedback}.
Now applying Proposition \ref{prop:fund-rel} and Corollary \ref{cor:fund-rel} to the a.c.s. $(\Omega,\cF,(\cF_t),\hat{\PP},\hat{W},\hat{\gamma},\bar{\bx}(t))$  with $$\hat{\gamma}(t)=
  \mu(t,\bar{\bx}(t),[\nabla v(I-B)^{1-\theta}](t,\bar{\bx}(t))(I-B)^\theta P g)$$
we obtain the required conclusions.
\end{proof}

\end{document}